\documentclass[a4paper,11pt]{amsart}
\usepackage{amsmath,amsthm,amsfonts,amssymb,mathrsfs}
\usepackage{enumerate}
\usepackage[colorlinks=true, linkcolor=blue, citecolor=magenta, menucolor=black]{hyperref}
\usepackage{geometry}
\usepackage[demo]{graphicx}
\usepackage[up]{caption}
\usepackage{subcaption}
\usepackage{pgf,tikz}
\usepackage[toc,page]{appendix}
\usetikzlibrary{arrows}
\usetikzlibrary{patterns}
\usepackage{color}
\usepackage{tikz-cd}

\numberwithin{equation}{section}
\theoremstyle{plain}
\newtheorem{thm}{Theorem}[section]

\newtheorem{lem}[thm]{Lemma}
\newtheorem{prop}[thm]{Proposition}
\newtheorem{cor}[thm]{Corollary}

\newtheorem{obs}[thm]{Observation}

\newtheorem{letterthm}{Theorem}

\theoremstyle{definition}
\newtheorem{defn}[thm]{Definition}
\newtheorem*{defn*}{Definition}

\newtheorem{exmp}[thm]{Example}

\newtheorem{rem}[thm]{Remark}

\newtheorem*{example}{Example}

\newtheorem{claim}[thm]{Claim}

\newcommand{\C}{\mathbb{C}}

\newcommand{\G}{\mathbf{G}}
\newcommand{\bH}{\mathbf{H}}
\newcommand{\N}{\mathbb{N}}

\newcommand{\Q}{\mathbb{Q}}
\newcommand{\R}{\mathbb{R}}

\newcommand{\Z}{\mathbb{Z}}

\newcommand{\cF}{\mathcal{F}}

\newcommand{\cM}{\mathcal{M}}
\newcommand{\cN}{\mathcal{N}}

\newcommand{\cS}{\mathcal{S}}

\newcommand{\cU}{\mathcal{U}}

\newcommand{\cZ}{\mathcal{Z}}

\newcommand{\tpi}{{\widetilde{\pi}}}
\newcommand{\txi}{{\widetilde{\xi}}}

\newcommand{\tsigma}{{\widetilde{\sigma}}}
\newcommand{\ttheta}{{\widetilde{\theta}}}

\newcommand{\tC}{{\widetilde{C}}}
\newcommand{\tE}{{\widetilde{E}}}
\newcommand{\tH}{{\widetilde{H}}}
\newcommand{\tM}{{\widetilde{M}}}
\newcommand{\tN}{{\widetilde{N}}}

\newcommand{\oH}{{\overline H}}

\newcommand{\Tr}{\operatorname{Tr}}
\newcommand{\tr}{\operatorname{tr}}

\newcommand{\Stab}{\operatorname{Stab}}
\newcommand{\Fix}{\operatorname{Fix}}

\newcommand{\conv}{\operatorname{conv}}

\newcommand{\ot}{\otimes}
\newcommand{\ovt}{\mathbin{\overline{\otimes}}}
\newcommand{\Aut}{\operatorname{Aut}}
\newcommand{\Hom}{\operatorname{Hom}}
\newcommand{\Ad}{\operatorname{Ad}}

\newcommand{\id}{\operatorname{id}}
\newcommand{\GL}{\operatorname{GL}}
\newcommand{\PGL}{\operatorname{PGL}}
\newcommand{\EL}{\operatorname{EL}}

\newcommand{\SL}{\operatorname{SL}}

\newcommand{\U}{\operatorname{U}}

\newcommand{\Prob}{\operatorname{Prob}}

\newcommand{\Sub}{\operatorname{Sub}}

\newcommand{\bary}{\operatorname{Bar}}
\newcommand{\rank}{\operatorname{rank}}
\newcommand{\Rad}{\operatorname{Rad}}

\newcommand{\dpr}{^{\prime\prime}}
\newcommand{\dd}{{\,\mathrm d}}

\newcommand{\eps}{\varepsilon}
\newcommand{\actson}{\curvearrowright}

\newcommand{\rC}{\operatorname{C}}

\newcommand{\Har}{\operatorname{Har}}

\newcommand{\Char}{\operatorname{Char}}
\newcommand{\PD}{\operatorname{PD}}

\allowdisplaybreaks

\begin{document}

\title[Charmenability of arithmetic groups of product type]{Charmenability of arithmetic groups of product type}

\author{Uri Bader}
\address{Faculty of Mathematics and Computer Science \\ The Weizmann Institute of Science \\ 234 Herzl Street \\ Rehovot 7610001 \\ ISRAEL}
\email{bader@weizmann.ac.il}
\thanks{UB is supported by ISF Moked 713510 grant number 2919/19}

\author{R\'emi Boutonnet}
\address{Institut de Math\'ematiques de Bordeaux \\ CNRS \\ Universit\'e Bordeaux I \\ 33405 Talence \\ FRANCE}
\email{remi.boutonnet@math.u-bordeaux.fr}
\thanks{RB is supported by a PEPS grant from CNRS and ANR grant AODynG, 19-CE40-0008}

\author{Cyril Houdayer}
\address{Universit\'e Paris-Saclay \\ Institut Universitaire de France \\ CNRS \\ Laboratoire de math\'ematiques d'Orsay \\ 91405 \\ Orsay  \\ FRANCE }
\email{cyril.houdayer@universite-paris-saclay.fr}
\thanks{CH is supported by Institut Universitaire de France and FY2019 JSPS Invitational Fellowship for Research in Japan (long term)}

\author{Jesse Peterson}
\address{Department of Mathematics \\ Vanderbilt University \\ 1326 Stevenson Center \\ Nashville \\ TN 37240 \\ USA}
\email{jesse.d.peterson@vanderbilt.edu}
\thanks{JP is supported by NSF Grant DMS \#1801125 and NSF FRG Grant \#1853989}

\subjclass[2010]{22D10, 22D25, 22E40, 37B05, 46L10, 46L30}
\keywords{Arithmetic groups; Characters; Irreducible lattices; Poisson boundaries; Semisimple algebraic groups; Tree automorphism groups; Unitary representations; von Neumann algebras}

\begin{abstract}
We discuss special properties of the spaces of characters and positive definite functions, as well as their associated dynamics,
for arithmetic groups of product type.
Axiomatizing these properties, we define the notions of {\em charmenability} and {\em charfiniteness} and study their applications to the
topological dynamics, ergodic theory and unitary representation theory of the given groups.
To do that, we study singularity properties of equivariant normal ucp maps between certain
von Neumann algebras.
We apply our discussion also to groups acting on product of trees.  
\end{abstract}

\maketitle

\section{Introduction and statements of the main results}

For a countable discrete group $\Gamma$, we consider the convex set $\PD_1(\Gamma)\subset \ell^\infty(\Gamma)$ consisting of normalized positive definite functions and endow it with the weak$^*$-topology (which coincides with the topology of pointwise convergence)
and the $\Gamma$-action associated with the conjugation action of $\Gamma$ on itself.
This is a compact convex $\Gamma$-space.
Its compact convex subset consisting of $\Gamma$-fixed points is denoted by $\Char(\Gamma)$ and its elements are called {\em characters}\footnote{Beware that in some texts the term ``character" is reserved for an extreme point in this set.} of $\Gamma$. The GNS representation $(\pi, H, \xi)$ associated with a character $\phi \in \Char(\Gamma)$ generates a tracial von Neumann algebra $M = \pi(\Gamma)\dpr$. Then $\phi \in \Char(\Gamma)$ is an extremal character if and only if $M = \pi(\Gamma)\dpr$ is a factor, that is, a von Neumann algebra with trivial center.

The problem of the classification of characters of higher rank lattices has seen important progress in the last fifteen years. It has also attracted a lot of attention because of its connection with the theory of Invariant Random Subgroups (IRS) (see e.g.\ \cite{7s12, AGV12, Ge14}). Bekka \cite{Be06} obtained a complete classification of characters of $\SL_n(\Z)$ for $n \geq 3$. This result was later extended by Peterson \cite{Pe14} to all higher rank lattices with property (T). Recently, Boutonnet-Houdayer \cite{BH19} strengthened these results and obtained a complete classification of {\em stationary characters} of higher rank lattices in simple Lie groups. 
We refer to \cite{CP13, PT13, Be19, BeF20, LL20} for other classification results for characters.


Before stating our main theorems, we first introduce some terminology.

\begin{defn}
A character $\phi$ on $\Gamma$ is called {\em amenable} if the corresponding GNS representation $(\pi,H)$ is amenable in the sense of \cite{Be89}, that is, $\pi \otimes \overline \pi$ weakly contains the trivial representation. It is called {\em von Neumann amenable} if $\pi(\Gamma)\dpr$ is moreover an amenable von Neumann algebra. It is called {\em finite} if $H$ is finite dimensional.
\end{defn}

Note that if $\Gamma$ is amenable, then any $\Gamma$-invariant compact convex subset of $\PD_1(\Gamma)$ contains a character, and every character of $\Gamma$ is von Neumann amenable. Conversely, a non-amenable group always contains a character that is not von Neumann amenable, namely the regular character $\delta_e$. In fact, if $\Gamma$ is non-amenable, any character supported on the amenable radical of $\Gamma$ is not amenable.

\begin{defn} \label{def:charmenable}
The group $\Gamma$ is said to be {\em charmenable} if it satisfies the following two properties:
\begin{enumerate}
\item Every non-empty compact convex $\Gamma$-invariant subset of $\PD_1(\Gamma)$ contains a character.
\item Every extremal character of $\Gamma$ is either supported on the amenable radical $\Rad(\Gamma)$ or von Neumann amenable.
\end{enumerate}
Moreover, $\Gamma$ is said to be {\em charfinite} if it also satisfies the following properties:
\begin{enumerate} \setcounter{enumi}{2}
\item $\Rad(\Gamma)$ is finite.
\item $\Gamma$ has a finite number of isomorphism classes of unitary representations in each given finite dimension.
\item Every amenable extremal character of $\Gamma$ is finite. 
\end{enumerate}
\end{defn}

As we will see in \S\ref{sec:charamenable}, charmenable and \ charfinite groups enjoy
remarkable properties pertaining to the structure of C$^*$-algebras associated with their unitary representations and the stabilizer structure of their ergodic and topological actions. In particular, we will see the following (see \cite{GW14} for the notion of URS):
\begin{itemize}
\item For any charmenable group $\Gamma$ with trivial amenable radical, every URS carries a $\Gamma$-invariant Borel probability measure and every non-amenable unitary $\Gamma$-representation weakly contains the left regular representation.
\item Furthermore, for any charfinite group $\Gamma$, every URS and every ergodic IRS is finite. If $\Gamma$ moreover has property (T) and trivial amenable radical, then every weakly mixing representation weakly contains the left regular representation.
\end{itemize}


The main results of the present paper prove charmenability or charfiniteness for certain arithmetic groups of product type. In view of the above remarks, these extend results obtained in \cite{BH19} for lattices in higher rank connected simple Lie groups (see for example \cite[Theorem D, Corollary F]{BH19}). To help distinguish between the two papers, let us specify that the strategy of \cite{BH19} based on a Nevo-Zimmer type theorem collapses for lattices in semi-simple groups in the presence of a rank one factor. This difficulty brings us to develop more abstract tools on non-commutative ergodic theory, leading naturally to the general notions of charmenable and charfinite groups. On the other hand, the approach of the present paper is inapplicable to lattices in simple groups.

Let us recall some definitions on arithmetic groups.

\begin{defn} \label{def:arith}
Let $K$ be a global field and ${\bf G}$ a connected non-commutative $K$-almost simple $K$-algebraic group.
Let $S$ be a (possibly empty, possibly infinite) set of non-archimedean inequivalent absolute values on $K$,
let $\mathcal{O}<K$ be the ring of integers and let $\mathcal{O}_S$ the corresponding localization,
that is,
\[ \mathcal{O}_S = \{ \alpha\in K\mid \forall s\in S, s(\alpha) \leq 1\}.  \]
Fix an injective $K$-representation $\rho:{\bf G} \to \GL_n$ and denote 
\[ \Lambda_S=\rho^{-1}(\GL_n(\mathcal{O}_S))\le {\bf G}(K). \]
The triple $(K,{\bf G},S)$ is said to be
\begin{itemize}
\item \emph{of a compact type}
if for every absolute value $v$ on $K$, the image of $\Lambda_S$ in ${\bf G}(K_v)$ is bounded,
\item \emph{of a simple type}
if there exists a unique absolute value $v$ on $K$ such that the image of $\Lambda_S$ in ${\bf G}(K_v)$ is unbounded
\item and \emph{of a product type} otherwise.
%
\end{itemize}
The triple $(K,{\bf G},S)$ is said to be \emph{of higher rank} if it is either of a product type or of a simple type and $\rank_{K_v}({\bf G}) \geq 2$.

A subgroup $\Gamma\le {\bf G}(K)$ is called \emph{$S$-arithmetic} 
if it is commensurable with $\Lambda_S$.
It is called \emph{arithmetic} if it is $S$-arithmetic for some $S$ as above
and we regard its type as the type of $(K,{\bf G},S)$.
\end{defn}

\begin{example}
Let $K = \Q$, ${\bf G} = \SL_n$ for $n \geq 2$ and $S \subset \mathcal P$ a (possibly empty, possibly infinite) set of primes. If $S \neq \emptyset$, then $\SL_n(\Z_S) \leq \SL_n(\Q)$ is an $S$-arithmetic group of product type.
\end{example}

\begin{letterthm} \label{thm:AG}
Let $K$ be a global field and ${\bf G}$ a connected non-commutative $K$-almost simple $K$-algebraic group.
If $\Gamma\leq {\bf G}(K)$ is an arithmetic subgroup of a product type
then $\Gamma$ is charmenable.

Assume further that there exists an absolute value $v$ on $K$ such that ${\bf G}(K_v)$ has property {\em (T)} and for which the image of $\Gamma$ in ${\bf G}(K_v)$ is unbounded.
If either $S$ is finite or ${\bf G}$ is simply connected then 
$\Gamma$ is charfinite.
\end{letterthm}

The proof of Theorem~\ref{thm:AG} will be given in \S\ref{sec:upgrade}.

The assumption that one of the factors has property (T) is not a necessary condition for prompting charmenability to charfiniteness. Indeed, using \cite[Theorem 2.6]{PT13}, we obtain the following result.

\begin{letterthm} \label{thm:Zp}
For every non-empty set of primes $S$, the group $\SL_2(\mathbb{Z}_S)$ is charfinite.
\end{letterthm}

The proof of Theorem~\ref{thm:Zp} will be given in \S\ref{ss:fdur}.

Let us point out that the case $\Gamma=\SL_2(\mathbb{Q})$ (that is, where $S$ in the above theorem is the set of all primes)
is particularly interesting, as this group has no non-trivial finite dimensional unitary representations.
It follows that the only extremal characters on this group are the regular and the trivial characters and that
every $\Gamma$-invariant compact convex subset of $\PD_1(\Gamma)$ contains a convex combination of these two characters.
However, $\Gamma$ is not finitely generated.
It will be very interesting to find a finitely generated charfinite simple infinite group. 
We expect certain Kac-Moody groups to satisfy all of these properties.

When $\Gamma$ is of a simple type and the corresponding absolute value is archimedean (e.g.\ $\Gamma=\SL_n(\mathbb{Z})$), the conclusion of Theorem \ref{thm:AG} still holds under the assumption that $\Gamma$ is of higher rank (e.g.\ $n\geq 3$), that is, $\Gamma$ is charfinite in this case.
See Corollary \ref{cor:BH19} for an exact formulation and see also Remark~\ref{rem:comingsoon}.
The following is a slight strengthening of the above example.

\begin{letterthm} \label{thm:SD}
For any $n\geq 3$, the group $\SL_n(\mathbb{Z}) \ltimes \mathbb{Z}^n$ is charmenable.
\end{letterthm}

The proof of Theorem~\ref{thm:SD} will be given in \S\ref{ss:proof}.

A fundamental concept in this paper is the notion of a $(G,N)$-von Neumann algebra $M$,
which is a choice of an equivariant normal ucp map $M\to N$,
where $G$ is an lcsc group and $M,N$ are $G$-von Neumann algebras.
In \S\ref{ss:charcrit} we will give general criteria for charmenability 
based on the notion of \emph{singularity} of a $(G,N)$-structure. 
The proofs of all theorems presented above will rely on these charmenability criteria.
The proofs of Corollary \ref{cor:BH19} and Theorem~\ref{thm:SD} 
will also rely heavily on \cite[Theorem B]{BH19}
which forms a noncommutative Nevo-Zimmer structure theorem for stationary actions on von Neumann algebras.
However, as pointed out in \cite{NZ97, NZ00}, such a structure theorem cannot hold for semisimple Lie groups admitting 
a rank one factor and therefore the method of \cite{BH19} could not be applied for proving our main theorem, Theorem~\ref{thm:AG}.
To overcome this conceptual difficulty we develop a new strategy which applies in the setting of lattices with dense projections.

\begin{defn}
Let $I$ be a finite set and $G_i$ be an lcsc group for each $i\in I$.
Let $G=\prod_{i\in I} G_i$ and $\Gamma\le G$ be a lattice.
We say that $\Gamma$ has {\em dense projections} if 
its image in $\prod_{i \neq i_0} G_i$ is dense, for every $i_0\in I$.
\end{defn}

For such a lattice with dense projections $\Gamma \le G$ we will consider in Theorem \ref{non-trivial G points} 
the structure of $(\Gamma,N)$-von Neumann algebras, where $N$ is the $L^\infty$-algebra of the Furstenberg-Poisson boundary of $G$.
This theorem will allow us to shift the discussion on $\Gamma$-dynamical systems to $G_i$-dynamical systems, where $G_i$ is one of the simple factors. From there we will use the special form of $G_i$, its parabolic subgroups and Mautner phenomenon to deduce the desired singularity property. This second half is based on Proposition \ref{FMM implies singular}.  

In fact, for the groups considered in this paper, the combination of Theorem \ref{non-trivial G points} and Proposition \ref{FMM implies singular} implies condition (a) from Proposition \ref{top crit}, which is our replacement of \cite[Theorem B]{BH19}. Note that in the setting of \cite{BH19}, the non-commutative Nevo-Zimmer theorem implies this condition (a), although the proof is different.

Note that in the setting of Theorem~\ref{thm:AG}, if $S$ is finite and $\G$ is simply connected then the group $\Gamma$ is a lattice with dense projections in $\prod_{v\in V} \G(K_v)$, where $V$ denotes the set of places of $K$ under which $\Gamma$ is unbounded.
This indeed holds by the strong approximation theorem, see \cite[Theorem II.6.8]{Ma91}.
In the proof of Proposition~\ref{SAF}, which is a main step towards proving Theorem~\ref{thm:AG}, we will explain 
how to reduce the general case to the case above, finite $S$ and simply connected $\G$,
in which Theorem \ref{non-trivial G points} is applicable.
We emphasize that, apart of invoking the strong approximation theorem, our work 
here does not rely at all on arithmeticity properties of the groups under consideration
and our choice of presenting Theorem~\ref{thm:AG} for arithmetic lattices rather than lattices with dense projections in the first place
is a matter of taste more than anything else.

Our next theorem deals with a geometric situation which generalizes a product of rank one groups over non-archimedean fields.

\begin{letterthm} \label{thm:Tree}
Let $n\geq 2$. For every $i=1,\ldots,n$, let $T_i$ be a bi-regular tree and denote by $\Aut^+(T_i)$ the group of bi-coloring preserving automorphisms of $T_i$. Let $\Gamma < \Aut^+(T_1) \times \cdots \times \Aut^+(T_n)$ be a cocompact lattice. Denote by $G_i$ the closure of the image of $\Gamma$ in $\Aut^+(T_i)$ and assume that $G_i$ is $2$-transitive on the boundary $\partial T_i$. Further assume that $\Gamma < G_1 \times \cdots \times G_n$ is with dense projections.  Then $\Gamma$ is charmenable.
\end{letterthm}

The proof of Theorem~\ref{thm:Tree} will be given in \S\ref{ss:lpt}.

Note that Theorem \ref{thm:Tree} implies that the finitely presented, torsion-free, simple groups constructed by Burger-Mozes in \cite{BM00} are charmenable. In particular, it follows that any non-trivial URS of such a group is necessarily supported on co-amenable subgroups. It remains an open problem to prove that the groups appearing in Theorem \ref{thm:Tree} are charfinite. The proof of Theorem \ref{thm:Tree} follows the same strategy as the one of Theorem \ref{thm:AG} except that we exploit the $2$-transitivity of each factor group $G_i$ in lieu of semisimplicity. 

\begin{rem}
It seems reasonable to expect that some or all of the groups in Theorem~\ref{thm:Tree} are actually charfinite,
but so far we are unable to prove property (5) in Definition~\ref{def:charmenable}.
The fact that these groups satisfy property (3) is well known and property (4) 
is established in Proposition~\ref{prop:ffdlt} below.
\end{rem}

\begin{rem} \label{rem:comingsoon}
In a sequel work, we will show that in Theorem~\ref{thm:AG}, the assumption that $\Gamma$ is of product type could be replaced by a higher rank assumption.
We will do this by combining the techniques developed in the current paper with the ones developed in \cite{BH19}. This will completely settle the question of charmenability for lattices in semisimple algebraic groups. 
\end{rem}

\subsubsection*{Acknowledgments}
We wish to thank Yair Glasner and Pierre-Emmanuel Caprace for providing us with the proof of Proposition~\ref{freeness2}.
We thank Pierre-Emmanuel Caprace also for finding a flawed claim we made in an earlier version of this paper and for various suggestions for improvements of our presentation. We are grateful to Adrian Ioana and Narutaka Ozawa for their valuable remarks. Part of this work was done when CH was visiting the University of Tokyo during January-July, 2020. He would like to thank Yasuyuki Kawahigashi and the Graduate School of Mathematical Sciences for their kind hospitality. Last but not least, we thank the anonymous referees for carefully reading our paper and providing useful remarks.

{
  \hypersetup{linkcolor=black}
  \tableofcontents
}

\section{Preliminaries}

In this section we collect various preliminary definitions and results.
\S\ref{ss:PDF} discusses positive definite functions 
and \S\ref{GactM} discusses group actions on operator algebras,
both are core concepts of this paper.
In \S\ref{ME} we discuss metric ergodicity and the Mautner property,
which will become important when discussing charmenability criteria in \S\ref{ss:charcrit}.

\subsection{Positive definite functions} \label{ss:PDF}

In this subsection we consider positive definite functions on a locally compact group $G$. The $L^p$-spaces over $G$ will be considered with respect to the Haar measure $\mu_G$.

Recall that a function $\phi\in L^\infty(G)$ is said to be positive definite if $\int_G (f^**f)\phi \dd\mu_G \geq 0$, for every $f\in L^1(G)$. A positive definite function is necessarily continuous, that is, agrees a.e.\ with a continuous function.
The set of all positive definite functions on $G$ is denoted $\PD(G)$ and we denote by $\PD_1(G)$ the subset of functions $\phi$ satisfying $\phi(e)=1$. We endow it with the subspace topology inherited from the weak*-topology on $L^\infty(G)$. 
In case $G$ is discrete,  $\PD_1(G)$ becomes a compact convex $G$-space.
The compact convex subset consisting of $G$-invariant points in $\PD_1(G)$ is denoted $\Char(G)$ and its elements are called characters. The extreme points of $\Char(G)$ are called extremal characters.

\begin{defn}
Let $\phi \in \PD(G)$. By definition, the associated {\em GNS triple} $(\pi_\phi,H_\phi,\xi_\phi)$ is the data of a unitary representation $\pi_\phi$ of $G$ on the Hilbert space $H_\phi$, together with a cyclic vector $\xi_\phi \in H_\phi$ satisfying $\langle \pi_\phi(g)\xi_\phi,\xi_\phi\rangle = \phi(g)$, for all $g \in G$. Such a triple is unique up to conjugation (i.e.\ up to an isomorphism of the Hilbert spaces, which intertwines the representations, and maps cyclic vector to cyclic vector). 

For $\phi\in \Char(G)$, $\pi_\phi$ extends to a unitary representation $\widetilde{\pi}_\phi$ of $G\times G$ on $H_\phi$
whose restriction to the left factor is $\pi_\phi$ and for which $\xi_\phi$ is invariant under the diagonal subgroup in $G\times G$.
\end{defn}

Every lcsc group has at least one character: the trivial character, namely the constant function~$1$.
Every non-trivial discrete group has at least one more character, the {\em regular character} $\delta_e$.
In general the trivial character might be the only character.

\begin{prop} \label{prop:trivcharlie}
Let $k$ be a local field and ${\bf G}$ a connected simply connected $k$-isotropic $k$-almost simple $k$-algebraic group
and denote $G={\bf G}(k)$.
Then $\Char(G)=\{1\}$. 
\end{prop}

This Proposition is essentially due to Segal and von Neumann who proved it for non-compact simple Lie groups, see \cite{SN50}\footnote{See the discussion in p.\ 2 of \cite{BG14} for some history of ideas.}.

\begin{proof}
Our argument is essentially based on Howe-Moore theorem.
Take a character $\phi$ on $G$ and use the GNS representation to get an isometric action of $G\times G$ on $H_\phi$. The stabilizer $H < G \times G$ of $\xi_\phi$ is non-compact as it contains the diagonal subgroup. By \cite[Theorem 6.2]{BG14}, we find that $H$ must contain a non-trivial factor of $G$. So $\xi_\phi$ is either $G \times 1$ or $1 \times G$ invariant. Using again the invariance under the diagonal, we find that $\xi_\phi$ is $G \times 1$-invariant, i.e. $\pi_\phi$-invariant. This gives $\phi = 1$.
\end{proof}

The assumption that ${\bf G}$ is $k$-isotropic is equivalent to the non-compactness of ${\bf G}(k)$
and it is essential.
Indeed, non-trivial compact groups always admit non-trivial characters.
The following is well known.
\begin{prop} \label{prop:compt}
Let $K$ be a compact group.
Then the extremal characters of $K$ are in one to one correspondence with its irreducible representations;
the correspondence is given by assigning to the irreducible representation $\pi$ the character $g \mapsto \tr_\pi(\pi(g))$,
where $\tr_\pi$ is the normalized trace associated with $\pi$.
The GNS construction associated with this character is the direct sum of $\dim(\pi)$ copies of $\pi$.

In particular, every character of $K$ is obtained by a summable convex combination of countably many extremal characters.
\end{prop}

\begin{defn}
Let $\phi \in \PD(G)$. We say that $\phi$ is
\begin{itemize}
\item {\em Compact} if $\pi_\phi$ is a compact representation, i.e.\ $\pi_\phi(G)$ is relatively compact in $\cU(H_\phi)$ for the strong operator topology.
\item {\em Amenable} if $\pi_\phi$ is amenable, i.e.\ there is a state $\Phi$ on $B(H)$ which is invariant under $\Ad(\pi_\phi(g))$, $g \in G$
or equivalently, $\pi\otimes \overline{\pi}$ weakly contains the trivial representation.
\item {\em von Neumann amenable} if $\pi_\phi(G)''$ is an injective von Neumann algebra, i.e.\ there is a conditional expectation $E: B(H_\phi) \to \pi_\phi(G)''$.
\end{itemize}
\end{defn}

A compact positive definite function is von Neumann amenable, hence amenable.
Note that a compact character is a character which factorizes through a character on the Bohr compactification of $G$. So by the previous proposition, any compact character is a countable convex combination of countably many extremal characters.

In the case where $\phi$ is a character, $\phi$ is von Neumann amenable if and only if there exists an $\Ad(\pi_\phi(G))$-central state on $B(H_\phi)$ such that $\Phi(x) = \langle x\xi_\phi,\xi_\phi\rangle$ for every $x \in \pi_\phi(G)\dpr$. In other words, $\Phi$ is an extension of $\phi$, which is normal on $\pi_\phi(G)\dpr$. Indeed, the existence of such a state implies the amenability property of $\pi_\phi(G)\dpr$, which is known to be equivalent to injectivity. On the other hand, if $\pi_\phi(G)\dpr$ is injective, then we can compose the conditional expectation $E: B(H_\phi) \to \pi_\phi(G)\dpr$ with the trace $\langle \, \cdot \, \xi_\phi,\xi_\phi\rangle$ on $\pi_\phi(G)\dpr$ to get the desired state extension.

We point out that the spaces of compact or (von Neumann) amenable PD-functions are not closed in general. For example, 
if $G$ is a non-amenable residually finite discrete group then the regular character, which is not amenable, lies in the closure of the compact characters. Nevertheless these sets are easily checked to be Borel sets. Moreover, we have the following convexity property.

\begin{lem}\label{convexity of amenable reps}
Let $\nu \in \Prob(\Char(G))$ and $t := \nu(\{\text{von Neumann amenable characters}\})$. Denote by $\phi := \bary(\nu)$, by $(H,\pi,\xi)$ the corresponding GNS triple, by $M := \pi(G)\dpr$ and by $\tau = \langle \, \cdot \, \xi,\xi\rangle$ the unique normal trace on $M$ that extends $\phi$.

Then there exists a projection $p \in M$ with trace at least $t$ such that $pMp$ is amenable. In particular, if $\nu$ is supported on the set of von Neumann amenable characters (i.e.\ $t = 1$) then $\bary(\nu)$ is von Neumann amenable.
\end{lem}
\begin{proof}
For simplicity, we denote by $X = \Char(G)$ and by $X_0$ the subset of von Neumann amenable characters.
Denote $\phi=\bary(\nu)$ and identify $(\pi,H,\xi)$ with (the cyclic subspace of) 
\[(\tpi,\tH,\txi) := \int_X^\oplus (\pi_\psi,H_\psi,\xi_\psi) \dd\nu(\psi).\]
It is shown in \cite[Lemma 4.1]{AB18} that $\tpi(G)\dpr \simeq \pi(G)\dpr = M$ (and we observe that this identification preserves the trace). So in the sequel we will rather denote by $M = \tpi(G)''$.

Denote by $p_0 = \mathbf{1}_{X_0} \in B(\tH)$ the orthogonal projection onto the subspace $\int_{X_0}^\oplus H_\psi \dd\nu(\psi) \subset \tH$. Then this projection lies inside $\tpi(G)' = M'$ and satisfies $\langle p_0\txi,\txi \rangle = t$. Moreover $Mp_0$ is contained in the amenable tracial von Neumann algebra $\int_{X_0}^\oplus \pi_\psi(G)\dpr \dd\nu(\psi)$, so $Mp_0$ is amenable as well.
Denote by $p \in \cZ(M) = \cZ(M')$ the central support of $p_0 \in M'$. Then $pM$ is amenable and we have $\tau(p) = \langle p\txi,\txi\rangle \geq \langle p_0\txi,\txi \rangle = t$, as desired.
\end{proof}

However the above convexity property doesn't hold for compact characters. For example, if $G = \Z$, then the regular character is not compact, but by Fourier transform, it is the Lebesgue average of the compact characters $\phi_z: n \in \Z \mapsto e^{i2\pi nz}$, $z \in [0,1]$.

\subsection{Group actions on operator algebras} \label{GactM}

In this paper, we will consider group actions on C$^*$-algebras and von Neumann algebras. 
Let $G$ be an lcsc group. 

By a {\em $G$-C*-algebra} we mean a C$^*$-algebra $A$ endowed with a continuous map $G \times A \to A$, called the action map, which induces an action of $G$ on $A$, to be denoted $G \actson A$, by C*-algebra automorphisms. 
Such an action $G \curvearrowright A$ induces a weak$^*$ continuous affine action of $G$ on the state space $\cS(A)$ defined by the formula $g\phi := \phi \circ g^{-1}$, for all $g \in G$, $\phi \in \cS(A)$. In particular, every probability measure $\mu \in \Prob(G)$ defines a convolution operator
\[\phi \in \cS(A) \mapsto \mu \ast \phi := \int_G g\phi \dd\mu(g) \in \cS(A).\]
A fixed point for this convolution operator is called a {\em $\mu$-stationary state} on $A$. We will denote by $\cS_\mu(A) \subset \cS(A)$ the closed convex subset of $\mu$-stationary states.

By a {\em $G$-von Neumann algebra} we mean a von Neumann algebra $M$ endowed with a map $G \times M \to M$
which is continuous with respect to the ultraweak topology on $M$ and which induces 
an action of $G$ on $M$ by von Neumann algebra automorphisms. 
Recall that the ultraweak topology is the weak-* topology when $M$ is identified with the dual of its pre-dual $M = (M_\ast)^*$
and note that in general a $G$-von Neumann algebra is {\em not} a $G$-C*-algebra.
Again, such an action defines by duality an affine action $G \actson M_*$, which is continuous for the norm topology on $M_*$. As in the C*-case, any probability measure $\mu \in \Prob(G)$ gives rise to a convolution operator on $M_*$, and a normal state $\phi$ on $M$ fixed by this operator is called a {\em $\mu$-stationary state on $M$}.
We say that the action $G \curvearrowright M$ is {\em ergodic} if the fixed point algebra $M^G := \{x \in M \mid gx = x, \forall g \in G\}$ is trivial. 

To avoid notational misinterpretation, unless otherwise specified, we will generically use the notation $\sigma$ to denote our actions.

By a regularization argument, any $G$-von Neumann algebra $M$ admits an ultraweakly dense C*-subalgebra $A$ on which the action is norm continuous, see the proof of \cite[Proposition XIII.1.2]{Ta03b}). Since we assume $G$ to be second countable, if $M$ has separable predual we may choose $A$ to be a separable C*-subalgebra. This passage to a $G$-C*-algebra parallels the choice of a {\em compact model} in classical ergodic theory.

In the other direction, given a $G$-C*-algebra $A$, we may extend the $G$-action on $A$ to a $G$-action on $A^{**}$ but unfortunately this action is not continuous in general. However, when one restricts to certain corners of $A^{**}$ it may be continuous. 
To clarify here, let us mention that $A^{**}$ identifies with a von Neumann algebra: the universal enveloping von Neumann algebra of $A$. Its ultraweak topology is the weak-* topology on $A^{**} = (A^*)^*$. 

\begin{prop}\label{continuous vn extension}
Let $A$ be a $G$-C*-algebra and $N$ be a $G$-von Neumann algebra. Consider a $G$-equivariant unital completely positive (ucp) map $E: A \to N$.

Extend $E$ to a normal ucp map on $A^{**}$ and extend the $G$-action on $A$ to a (non-continuous) action on $A^{**}$. Denote by $z \in A^{**}$ the central support projection of $E$, i.e.\ the smallest projection in $\cZ(A^{**})$ such that $E(z) = 1$. Then $z$ is $G$-invariant and the $G$-action on $zA^{**}$ is a continuous von Neumann algebraic action.
\end{prop}

\begin{proof} The $G$-invariance of $z$ is clear from the definition; only the continuity statement has to be checked.
One can prove this fact by identifying $zA^{**}$ with the ``Stinespring von Neumann algebra'' of $E$, and checking that the action can be unitarily implemented in the Stinespring representation, and hence extends to a continuous action on $zA^{**}$. 

The more concise argument we provide was communicated to us by Narutaka Ozawa. Denote by $M := zA^{**}$. Denote by $L$ the set of all normal linear functionals of the form $x \in M \mapsto (\psi \circ E)(axb) \in \C$, for $\psi \in N_*$, $a,b \in A$. 

{\bf Claim.} $L$ spans a norm dense subspace in the predual $M_*$.

We will use Hahn-Banach theorem. Let $x \in M$ be such that $\phi(x) = 0$ for all $\phi \in L$. Let us show that $x = 0$. The assumption is that $\psi(E(axb)) = 0$ for all $\psi \in N_*$, $a,b \in A$. Thus $E(axb) = 0$, for all $a,b \in A$, and by normality, this is also true for all $a,b \in A^{**}$. So in particular, we have $E(ux^*xu^*) = 0$ for every unitary $u \in A^{**}$. If $x \neq 0$, then we may find a nonzero spectral projection $p$ of $x^*x$ and a positive scalar $\alpha$ such that $\alpha p \leq x^*x$. Then $E(upu^*) = 0$, for all unitary $u \in A^{**}$ and thus $E$ vanishes on the central projection $z_0 := \bigvee_{u \in \cU(A^{**})} upu^*$. But $x \in M = zA^{**}$, so $p \leq z$, and hence $z_0 \leq z$. In addition $z_0 \neq 0$, so the projection $z-z_0$ is strictly less than $z$ and still $E(z-z_0) = 1$, contradicting the definition of the central support $z$. This proves the claim.

Since the $G$-actions on $N_*$ and on $A$ are norm continuous, one checks that for every $\phi \in L$, the map $g \in G \mapsto g\phi \in M_*$ is norm continuous. By the claim, the action $G \actson M_*$ is norm continuous as well, and this is known to be equivalent to the action $G \actson M$ being a continuous von Neumann algebraic action.
\end{proof}

More generally, $G$-invariant corners of $A^{**}$ on which the action is continuous have been studied by Ikunishi, see \cite{Ik87}.

Let us now give two results about fixed point algebras in stationary von Neumann algebras. We recall that a probability measure $\mu$ on an lcsc group $G$ is {\em generating} if its support generates $G$ as a semigroup. 

\begin{prop}\label{prop:stationary}
Let $G$ be an lcsc group with a generating probability measure $\mu \in \Prob(G)$ and $M$ be a $G$-von Neumann algebra with a faithful $\mu$-stationary state $\phi \in M_*$. The following facts hold true.
\begin{enumerate}
\item There exists a unique $\phi$-preserving normal conditional expectation $E_\mu : M \to M^G$. 
\item  Every $\mu$-stationary normal state $\psi \in M_\ast$ satisfies $\psi = \psi \circ E_\mu$. In particular, if the action $G \curvearrowright M$ is ergodic, then $\phi$ is the only $\mu$-stationary normal state on $M$.
\item Let $A \subset M$ be any ultraweakly dense unital $G$-$\rC^*$-subalgebra. Then $G \curvearrowright M$ is ergodic if and only if $\phi_{|A} \in \cS_\mu(A)$ is an extreme point.
\end{enumerate}
\end{prop}

\begin{proof}
Given $(G,\mu)$ and $M$ as in the statement, define the convolution ucp map
\[T_\mu: x \in M \mapsto \check{\mu} \ast x = \int_G \sigma^{-1}_g(x) \dd\mu(g) \in M.\]
Since $\mu \ast \phi = \phi$, we have $\phi \circ T_\mu = \phi$. Since $\phi \in M_\ast$ is faithful, this implies that $T_\mu : M \to M$ is a faithful normal ucp map. Next, choose a non-principal ultrafilter $\omega \in \beta(\N) \setminus \N$ and define 
\[E_\mu: x \in M \mapsto \lim_{n \to \omega} \frac1n \sum_{k = 1}^n T_\mu^n(x) \in M.\]
Here the limit is meant for the ultra-weak topology. Observe that $E_\mu$ is a ucp map on $M$, it is idempotent and its image is the set of elements invariant under $T_\mu$. 

(1) Let $x \in M$ such that $E_\mu (x) = x$. Then $T_\mu(x) = x$ and since $\phi = \phi \circ T_\mu$, we find that
\begin{align*} \int_G \|x - \sigma_g^{-1}(x)\|_\phi^2 \dd\mu(g) & = \int_G \left(\phi(x^*x) - 2\Re ( \phi(x^*\sigma_g^{-1}(x))) + \phi(\sigma_g^{-1}(x^*x)) \right) \dd\mu(g)\\
& = \Vert x \Vert_\phi^2 - 2\Re (\phi(x^*T_\mu(x))) + \phi \circ T_\mu(x^*x) = 0.
\end{align*}
This implies that $\sigma_{g}^{-1}(x) = x$ for $\mu$-almost every $g \in G$. Since $\mu$ is generating and $G \curvearrowright M$ is continuous, it follows that $x \in M^G$. Therefore, $E_\mu : M \to M^G$ is a conditional expectation. Since $\phi \circ E_\mu = \phi$ and $\phi$ is a faithful normal state this implies that $E_\mu$ is also faithful and normal; it is the unique $\phi$-preserving conditional expectation onto $M^G$.

(2) For every $\mu$-stationary normal state $\psi \in M_\ast$, we have $\psi = \psi \circ T_\mu$. So the formula $\psi = \psi \circ E_\mu$ follows from the concrete formula defining $E_\mu$. If the action is ergodic then $\psi(x)1 = E_\mu(x) = \phi(x)1$ for every $x \in M$, showing that $\psi = \phi$. 

(3) Assume first that the action is ergodic. If $\psi \in \cS_\mu(A)$ is a positive linear functional such that $\psi \leq \phi$, then $\psi$ extends continuously to a normal linear functional on $M$, which must be proportional to $\phi$ thanks to (2). This implies that $\phi_{|A} \in \cS_\mu(A)$ is an extreme point.

Conversely, assume that $\phi_{|A} \in \cS_\mu(A)$ is an extreme point. Take an element $p \in M^G$, with $0 \leq p \leq 1$.
Denote by $(L^2(M),L^2(M)_+,J)$ the standard form of $M$, and by $\xi \in L^2(M)_+$ the unique positive vector implementing $\phi$ (see \cite{Ha73}). Define a linear functional $\psi \in M_\ast$ by the formula
\[\psi(x) = \langle x JpJ \xi, \xi\rangle, \text{ for every } x \in M.\]
We claim that $\psi$ is $\mu$-stationary as well. In fact, since $\phi \circ E_\mu = \phi$ we have $e_\mu(\xi) = \xi$, where $e_\mu$ is the orthogonal projection $L^2(M) \to L^2(M^G)$ corresponding to the conditional expectation $E_\mu$. Indeed this equivalence follows from the fact that $e_\mu$ maps positive vectors to positive vectors, and $e_\mu(\xi)$ implements $\phi \circ E_\mu$ (thanks to the formula $e_\mu x e_\mu = E_\mu(x)e_\mu$, for all $x \in M$).

The projection $e_\mu$ commutes with $J$ and since $p$ is $G$-invariant, we have $e_\mu p = p e_\mu$. So for every $x \in M$,
\[\psi(x) = \langle xJpJe_\mu(\xi),e_\mu(\xi)\rangle = \langle e_\mu x e_\mu JpJ\xi ,\xi\rangle = \langle  E_\mu(x)e_\mu JpJ\xi ,\xi\rangle = \psi \circ E_\mu(x).\] 
This shows that indeed $\psi$ is $\mu$-stationary. Moreover, it is obvious that $0 \leq \psi \leq \phi$, so by extremality of $\phi_{|A} \in \mathcal S_\mu(A)$, $\psi$ must be proportional to $\phi$ on $A$, and hence on $M$ (by ultraweak continuity): $\psi = c \phi$ for some $c \in [0,1]$. 
This implies that $\langle JpJ x\xi, y\xi\rangle = c \langle x\xi , y\xi\rangle$ for every $x,y \in M$. Since $\phi$ is faithful, $\xi$ is a cyclic vector and hence $JpJ = c 1 \in \C 1$ and so $M^G = \C 1$.
\end{proof}

\begin{lem}\label{lem:structure-product}
Let $G = G_1 \times G_2$ be the product of two lcsc groups. Choose generating measures $\mu_1 \in \Prob(G_1)$, $\mu_2 \in \Prob(G_2)$, and denote by $\mu = \mu_1 \ot \mu_2 \in \Prob(G)$ the product measure on~$G$.

Let $M$ be a $G$-von Neumann algebra with a faithful normal $\mu$-stationary state $\phi \in M_*$. Choose $i\neq j \in \{1,2\}$. The following facts hold true.
\begin{enumerate}
\item $\phi$ is $\mu_i$-stationary.
\item The unique $\phi$-preserving normal conditional expectation $E_{i} : M \to M^{G_i}$ is $G_j$-equivariant.
\item If $\phi$ is not $G_j$-invariant, $\phi_{|M^{G_i}}$ is not $G_j$-invariant. In particular, we have $M^{G_i} \neq \C 1$. 
\end{enumerate}
\end{lem}

\begin{proof}
(1) Note that $\mu \ast \mu_j = \mu_j \ast \mu$. Then we have 
\[\mu_j \ast \phi = \mu_j \ast \mu \ast \phi = \mu \ast (\mu_j \ast \phi).\]
This shows that $\mu_j \ast \phi$ is a $\mu$-stationary normal state. Since $\mu_j \ast \phi_{|M^G} = \phi_{|M^G}$, Proposition \ref{prop:stationary}(2) implies that $\mu_j \ast \phi = \phi$.

(2) Take $g \in G_j$. Then $g^{-1}\phi$ is still $\mu_i$-stationary, so by Proposition \ref{prop:stationary}.(2) we deduce that $g^{-1}\phi = (g^{-1}\phi) \circ E_i$. This rewrites as $\phi = \phi \circ (gE_ig^{-1})$. So $gE_ig^{-1}$ is a $\phi$-preserving conditional expectation onto $M^{G_i}$. By uniqueness, it coincides with $E_i$.

(3) follows trivially from (2).
\end{proof}

\subsection{Metric ergodicity and the Mautner property} \label{ME}

Metric ergodicity is an important tool that we use.
We recall its definition.

\begin{defn} \label{def:ME}
Let $G$ be an lcsc group and $B$ a $G$-Lebesgue space.
The action of $G$ on $B$ is called {\em metrically ergodic} if every measurable $G$-map from $B$ into a separable metric $G$-space $X$ on which $G$ acts continuously by isometries is essentially constant, equal to a $G$-fixed point.
We will say that $B$ is $G$-metrically ergodic.
\end{defn}

We refer to \cite[Section 2]{BF14} for examples and further extensions of this notion. 
For homogeneous spaces, metric ergodicity is closely related to the Mautner phenomenon.

\begin{defn}
Let $P$ be an lcsc group and $A \leq P$ a closed subgroup.
The pair $(P,A)$ is said to have the {\em Mautner property}
if for every continuous action of $P$ on a metric space $X$ by isometries, every point $x \in X$ which is $A$-invariant is $P$-invariant. 
\end{defn}

The following is an immediate consequence of \cite[Lemma~6.3]{BG14}.

\begin{lem}\label{mautner ME}
Let $P$ be an lcsc group and $A\le P$ a closed subgroup.
Endow $P/A$ with the unique $P$-invariant measure class.
Then the action of $P$ on $P/A$ is metrically ergodic if and only if 
the pair $(P,A)$ has the Mautner property.
\end{lem}

\begin{defn} \label{def:RMP}
Let $G$ be a topological group and $P\le G$ a closed subgroup. We will say that 
\begin{itemize}
\item $P$ has the {\em relative Mautner property} in $G$
if for every $g\in G$, the pair $(P,P \cap gPg^{-1})$ has the Mautner property. 
\item $P$ is {\em stably self normalizing} in $G$ if 
every intermediate closed subgroup $P\le Q \le G$ is its own normalizer in $G$.
\end{itemize}
\end{defn}

Note that both the relative Mautner property and the stably self normalizing property of $P$ in $G$
are conjugation invariant: if $P$ has it then also $gPg^{-1}$ for any $g\in G$. 

\begin{lem} \label{lem:RMP}
Let $G$ be an lcsc group and $P\le G$ a closed subgroup. Consider an intermediate closed subgroup $P \le H \le G$.
\begin{enumerate}
\item If $P$ is stably self normalizing or has the relative Mautner property in $G$, then the same is true in $H$.
\item If $P$ satisfies both of these properties, then for every $g \in H$, the pair $(H,P \cap gPg^{-1})$ has the Mautner property.
\end{enumerate}
\end{lem}
\begin{proof}
(1) is obvious.

(2) We may assume $H = G$. By an obvious transitivity property, we only need to prove that the pair $(G,P)$ has the Mautner property. Let $X$ be a metric space on which $G$ acts continuously by isometries and let $x\in X$ be a $P$-fixed point.
Denote by $S$ the stabilizer of $x$. Note that $(gPg^{-1},P \cap gPg^{-1})$ has the Mautner property, so $x$ must be fixed by $gPg^{-1}$, for every $g \in G$. So the closed group $Q$ generated by $\bigcup_{g \in G} gPg^{-1}$ is contained in $S$.
Since $Q$ is a normal subgroup of $G$ which contains $P$, the stably self-normalizing condition implies that $Q = G$. Hence $S = G$, as desired.
\end{proof}

\begin{exmp}\label{Mautner algps}
Let $k$ be a local field and ${\bf G}$ a connected non-commutative $k$-isotropic $k$-almost simple $k$-algebraic group.
Let ${\bf P}$ be a minimal $k$-parabolic subgroup.
We consider the subgroup ${\bf G}(k)^+$ discussed in \cite[Sections I.1.5 and I.2.3]{Ma91}
and recall that it is the image of $\tilde{\bf G}(k)$ under the covering map $\tilde{\bf G}\to {\bf G}$,
where $\tilde{\bf G}$ is the simply connected cover of ${\bf G}$, 
see \cite[Theorem I.2.3.1(a) and Proposition I.1.5.5]{Ma91}.
Let $G$ be an intermediate closed subgroup ${\bf G}(k)^+ \le G \le {\bf G}(k)$
and set $P=G\cap {\bf P}(k)$.
Note that by \cite[Proposition 20.5]{Bo91} we have a natural identification ${\bf G}(k)/{\bf P}(k)={\bf G}/{\bf P}(k)$.
We claim that $G$ acts transitively on ${\bf G}/{\bf P}(k)$ with stabilizer $P$ and $P$ is stably self normalizing and it has the relative Mautner property
in $G$.
We claim further that for every intermediate subgroup $P\le H\le G$, we have $H=G\cap {\bf Q}(k)$
for some intermediate $k$-parabolic subgroup ${\bf P}\le {\bf Q}\le {\bf G}$.

The fact that $G$ acts transitively on ${\bf G}(k)/{\bf P}(k)$ follows from 
the fact that already ${\bf G}(k)^+$ does.
Indeed, ${\bf G}(k)^+{\bf P}(k)={\bf G}(k)$ by \cite[Proposition I.1.5.4(vi)]{Ma91}.
It is now also clear that the stabilizer of the base point is $P$.

Next we show that $P$ has the relative Mautner property
in $G$.
We fix $g\in G$ and consider the corresponding conjugation of ${\bf P}$, ${\bf P}^g$.
Using \cite[Corollary 20.7(i)]{Bo91} we find a maximal $k$-split torus ${\bf T}\subset {\bf P}\cap {\bf P}^g$
and using \cite[Theorem 22.6(i),(ii)]{Bo91} we lift ${\bf P}$ and ${\bf T}$ to a $k$-parabolic $\tilde{\bf P} \subset \tilde{\bf G}$ and a maximal 
$k$-split torus $\tilde{\bf T}\subset \tilde{\bf P}$.
We let ${\bf U}$ and $\tilde{\bf U}$ be the corresponding unipotent radicals of ${\bf P}$ and $\tilde{\bf P}$,
which are defined over $k$ by \cite[Theorem 20.5]{Bo91}.
We set $\tilde{T}=\tilde{\bf T}(k)$ and $\tilde{U}=\tilde{\bf U}(k)$.
It is a standard fact, the classical Mautner phenomenon, that the pair
$(\tilde{T}\tilde{U},\tilde{T})$ has the Mautner property.
Let us explain why this implies that $(P,P \cap P^g)$ has the Mautner property.  
Thanks to Lemma \ref{mautner ME}, it is enough to consider the equivariant map $\tilde{T}\tilde{U}/\tilde{T} \to P/P \cap P^g$,
induced by the natural homomorphism $\tilde{T}\tilde{U}\to P$, and to observe that $\tilde{T}\tilde{U}$ acts transitively on $P/P \cap P^g$.
In fact, we may consider further the map $P/P \cap P^g\to {\bf P}(k)/{\bf P}(k)\cap {\bf P}^g(k)$
and note that $U={\bf U}(k)$ acts transitively on the latter by the Bruhat decomposition \cite[Theorem 21.15]{Bo91}
and the map $\tilde{U}\to U$ is surjective by \cite[Proposition 22.4(ii)]{Bo91}.
Thus indeed, $P$ has the relative Mautner property
in $G$.

We now let $H$ be an intermediate subgroup $P\le H\le G$.
We let $\tilde{H}$ be its preimage in $\tilde{\bf G}(k)$ and note that $\tilde{\bf P}(k)\le \tilde{H}$.
By \cite[Theorem 21.15]{Bo91} we have that $(\tilde{\bf G}(k),\tilde{\bf P}(k),N_{\tilde{\bf G}}(\tilde{\bf T})(k),S)$ is a Tits system where 
$S$ is the associated set of generators of the corresponding Weyl group.
We conclude that $\tilde{H}=\tilde{\bf Q}(k)$ for some $k$-parabolic subgroup $\tilde{\bf Q}$ in $\tilde{\bf G}$
and it is self normalizing.
Since $\tilde{\bf G}(k)$ acts transitively on $G/H$, as it acts transitively on $G/P$,
we get that there is a unique $\tilde{\bf G}(k)$-equivariant isomorphism between $\tilde{\bf G}(k)/\tilde{\bf Q}(k)$
and $G/H$.
By \cite[Theorem 22.6(i)]{Bo91} there exists a $k$-parabolic subgroup ${\bf Q}$ in ${\bf G}$ corresponding to $\tilde{\bf Q}$,
thus $\tilde{\bf G}(k)/\tilde{\bf Q}(k)$ and ${\bf G}(k)/{\bf Q}(k)$ are isomorphic as $\tilde{\bf G}(k)$ spaces,
and we conclude having a unique ${\bf G}(k)^+$-equivariant isomorphism between ${\bf G}(k)/{\bf Q}(k)$
and $G/H$. As ${\bf G}(k)^+$ is normal in $G$, we get that $G$ acts by conjugation on the set of all such ${\bf G}(k)^+$-equivariant isomorphisms,
thus this unique isomorphism must be $G$-invariant for the conjugation action, equivalently it is $G$-equivariant.
It follows that indeed, $H=G\cap {\bf Q}(k)$.
Moreover, it follows that $G/H$ has no non-trivial $G$-equivariant self maps, thus $H$ is self normalizing.

%
\end{exmp}

\begin{exmp}\label{Mautner trees}
Consider a thick simplicial tree $T$ and a group $G < \Aut(T)$ acting co-compactly on $T$. Denote by $\partial T$ the visual boundary of the tree, and assume that the action of $G$ on $\partial T$ is $2$-transitive.

Fix $\xi \in \partial T$, and denote by $P$ the stabilizer of $\xi$ in $G$. Then $G/P$ is homeomorphic with $\partial T$. 
Since $G$ acts $2$-transitively on $G/P$, we have a decomposition $G = P \sqcup PgP$, for any $g \in G \setminus P$.
In particular, $P$ is a maximal proper subgroup in $G$. So its normalizer is either equal to $P$ or to $G$, and the later case is impossible under the $2$-transitivity assumption. In conclusion $P$ is stably self normalizing. The next lemma ensures the Mautner condition.
\end{exmp}

\begin{lem}\label{lem:Mautner trees}
Keep the notation from Example \ref{Mautner trees}. Then $(G,P)$ has the relative Mautner property.
\end{lem}
\begin{proof}
We first point out that the $2$-transitivity assumption implies that $G$ has no fixed point in $T$. 
Indeed, otherwise its closure in $\Aut(T)$ is compact, thus so is its unique non-diagonal orbit in $\partial T \times \partial T$
and it follows that the diagonal is open, thus $\partial T$ is discrete, contradicting the thickness assumption.
It follows that $G$ contains a hyperbolic element.

Fix $g \in G$ and take a continuous isometric action $P \actson X$ on a metric space $X$. Assume that $x \in X$ is a $P \cap gPg^{-1}$-invariant point. We need to prove that it is $P$-invariant. Obviously, we may assume that $g \notin P$. In this case we observe that the $P$-orbit of $g\xi$ is open in $\partial T$, equal to $\partial T \setminus \xi$. In particular, the orbit map $p \in P \mapsto pg\xi \in \partial T$ is open.

Take $h \in P$ and $\eps > 0$. Consider the neighborhood $U:= \{pg\xi \mid p \in P, d(px,x) < \eps\}$ of $g\xi$ in $\partial T$. Since $G$ contains a hyperbolic element, it contains a hyperbolic element $k$ with axis $[g\xi,\xi]$, by $2$-transitivity. Then $k \in P \cap gPg^{-1}$, and we may find $n \in \Z$ such that $k^nhg\xi \in U$.
By definition of $U$, there exists $p \in P$ such that $d(px,x) < \eps$ and $pg\xi = k^nhg\xi$. In this case, the element $p^{-1}k^nh$ belongs to $P \cap gPg^{-1}$, and therefore fixes $x$. We may then compute:
\[d(hx,x) = d(hx,k^{-n}x) = d(k^nhx,x) \leq d(k^nhx,px) + d(px,x) < 0 + \eps.\]
Since $\eps$ is arbitrarily small, $h$ must fix $x$.
\end{proof}

\section{Charmenable and charfinite groups} \label{sec:charamenable}

In this section, we discuss charmenable and charfinite groups as defined in Definition~\ref{def:charmenable}.
In \S\ref{ss:prop} we list formal consequences of these definitions, in \S\ref{ss:unitary} we discuss unitary representations of such groups and in \S\ref{ss:perm} we consider some of their
permanence properties.
Giving actual criteria for charmenability and charfiniteness is the core of this work and will be taken in later sections (\S\ref{ss:charcrit}
and \S\ref{sec:upgrade}). 
At this stage we will just state the obvious:

\begin{obs}
Every amenable group is charmenable and every finite group is charfinite.
Further, an amenable group is charfinite if and only if it is finite. 
\end{obs}

\subsection{Properties of charmenable and charfinite groups} \label{ss:prop}

The following lemma will be often used without mention.

\begin{lem}\label{soft charm}
Every character of a charmenable group $\Gamma$ is a convex combination of a von Neumann amenable character and a character supported on $\Rad(\Gamma)$.
In particular, 
the set of characters supported on $\Rad(\Gamma)$ is a face of $\Char(\Gamma)$
and its complement set consists of amenable characters.

Furthermore, if $\Gamma$ is charfinite, then the GNS representation of an amenable character contains a finite dimensional subrepresentation.
\end{lem}
\begin{proof}
The lemma holds trivially if $\Gamma$ is amenable, thus we assume that this is not the case.
In particular, the characters supported on $\Rad(\Gamma)$ are non-amenable.
We denote the set of such characters $\Char_{\Rad}(\Gamma)$.

Given $\phi\in\Char(\Gamma)$ we get by Choquet's representation theorem $\nu\in \Prob(\Char(\Gamma))$ supported on the extreme points 
such that $\phi=\bary(\nu)$.
By the assumption that $\Gamma$ is charmenable we have that $\nu$ is a convex combination of $\nu_1$ and $\nu_2$,
where $\nu_1$ is supported on $\Char_{\Rad}(\Gamma)$ and $\nu_2$ is supported on the set of von Neumann amenable characters.
We conclude that $\phi$ is a convex combination of $\phi_1$ and $\phi_2$, where $\phi_i=\bary(\nu_i)$.
Clearly, $\phi_1\in \Char_{\Rad}(\Gamma)$ and by Lemma \ref{convexity of amenable reps},
$\phi_2$ is von Neumann amenable.
This proves the first claim.

If $\phi_2=0$ then $\phi=\phi_1\in \Char_{\Rad}(\Gamma)$.
Otherwise, the GNS representation of $\phi$ contains the GNS representation of $\phi_2$ and is thus amenable.
We get that $\phi$ is amenable if and only if it is not in $\Char_{\Rad}(\Gamma)$ and conclude that, indeed, $\Char_{\Rad}(\Gamma)$ is a face
whose complement consists of amenable characters.

Finally, if $\Gamma$ is charfinite and $\phi$ is amenable then the GNS representation of $\phi_2$, thus also of $\phi$,
contains a finite dimensional subrepresentation as $\nu_2$ is atomic and its atoms consist of finite characters. 
Indeed, $\nu_2$ is supported on a countable set of finite characters.
\end{proof}

The next proposition is a special case of Proposition~\ref{prop:IRS} and Proposition~\ref{prop:URS}.
Nevertheless we state it here for its importance and the clarity of its proof.

\begin{prop} \label{prop:NST}
Every normal subgroup $N \lhd \Gamma$ of a charmenable group is amenable or co-amenable in $\Gamma$.
If further $\Gamma$ is a charfinite group then $N$ is finite or of finite index in $\Gamma$.
\end{prop}
\begin{proof}
If $N$ is non-amenable, then $\phi:= \chi_N$ is not supported on $\Rad(\Gamma)$. 
Thus $\pi_\phi$ is an amenable representation and we get that $\Gamma/N$ is an amenable group.
If $\Gamma$ is charfinite and $N$ is infinite then also $\phi$ is not supported on $\Rad(\Gamma)$, which is finite,
thus again $\pi_\phi$ is amenable, hence it contains a finite dimensional subrepresentation.
However, $\pi_\phi$ is the regular representation of $\Gamma/N$ and so it follows that indeed, $\Gamma/N$ is finite.
\end{proof}

We denote by $\Sub(\Gamma)$ the space consisting of all subgroups of $\Gamma$
and endow it with the Chabauty topology.
This is a compact space on which $\Gamma$ acts by conjugation.
An IRS of $\Gamma$ is a $\Gamma$-invariant probability measure on $\Sub(\Gamma)$ (see \cite{AGV12}).

\begin{prop} \label{prop:IRS}
Let $\Gamma$ be a charfinite group and assume that $\Gamma$ acts ergodically on the probability space $(X,\mu)$ preserving the measure $\mu$.
Then either $X$ is essentially finite or the stabilizer of a.e.\ point of $X$ is contained in $\Rad(\Gamma)$.
In particular, every ergodic IRS of $\Gamma$ is finite.
\end{prop}

\begin{proof}
We assume $X$ is not essentially finite and consider the character $\phi(g)=\mu(\Fix(g))$.
We note that the GNS representation associated with $\phi$ is a sub-representation of $L^2(R)$ where $R\subset X\times X$ is the orbit equivalence relation endowed with
the $\mu$-integration of the counting measures on the fibers of the first coordinate projection.
By \cite[Proposition~3.1]{PT13}, this representation is weakly mixing and we deduce by charfiniteness that $\phi$ is indeed supported on $\Rad(\Gamma)$.
We note that the IRS associated with $X$ via the stabilizer map $X\to \Sub(\Gamma)$ is finite, as there are only finitely many subgroups in $\Rad(\Gamma)$.
The last bit of the proposition follows by the fact that every IRS of $\Gamma$ can be obtained as the image such a stabilizer map.
\end{proof}

Recall that a URS of $\Gamma$ is a minimal $\Gamma$-invariant subset of $\Sub(\Gamma)$ (see \cite{GW14}).

\begin{prop}\label{prop:URS}
Let $\Gamma$ be a charmenable group and $X$ a URS of $\Gamma$. 
Then either
every $H\in X$ is contained in $\Rad(\Gamma)$ or
$X$ carries a $\Gamma$-invariant probability measure.
Furthermore, if $\Gamma$ is charfinite then $X$ is finite.
\end{prop}
\begin{proof}
The map $\theta: H \in \Sub(\Gamma) \mapsto \mathbf{1}_H \in \PD_1(\Gamma)$ is continuous and $\Gamma$-equivariant. The push-forward of measures, together with the barycenter map $\Prob(\PD_1(\Gamma)) \to \PD_1(\Gamma)$ further yield a continuous affine $\Gamma$-map 
\[\ttheta : \Prob(X) \to \Prob(\PD_1(\Gamma)) \to \PD_1(\Gamma).\]
By charmenability, the closed convex $\Gamma$-subset $\ttheta(\Prob(X)) \subset \PD_1(\Gamma)$ contains a fixed point $\phi = \ttheta(\mu)$. By definition of $\ttheta$, we have $\phi(g) = \mu(\{H \in X \mid g \in H\})$. We observe that the GNS representation $\pi_\phi$ of $\phi$ is a sub-representation of the direct integral representation $\tpi$ on the space
\[\mathcal K := \int_X^\oplus \ell^2(\Gamma/H) \dd\mu(H).\]
If $\phi$ is supported on $\Rad(\Gamma)$ then $\mu$-almost every $H\in X$ is contained in $\Rad(\Gamma)$. Since $X$ is a URS, we find in this case that every $H \in X$ is contained in $\Rad(\Gamma)$.
We now assume this is not the case.
Thus $\pi_\phi$ is amenable by Lemma~\ref{soft charm}
and we get that the representation $\tpi$ on $\mathcal K$ is amenable as well: there exists a state $\Phi$ on $B(\mathcal K)$ which is invariant under conjugacy by elements $\tpi(g)$. 
Observe moreover that there is a $\Gamma$-equivariant *-homomorphism
\[\alpha: f \in C(X) \mapsto \int_X^\oplus f_H \dd\mu(H) \in B(\mathcal K),\]
where $f_H \in \ell^\infty(\Gamma/H)$ is defined by $f_H: \bar{g} \in \Gamma/H \mapsto f(gHg^{-1})$.
The composition $\Phi \circ \alpha$ is a $\Gamma$-invariant state on $C(X)$, i.e.\ a $\Gamma$-invariant Borel probability measure on $X$.

Finally, if $\Gamma$ is charfinite then both possibilities imply that $X$ is finite:
the first one by the finiteness of $\Rad(\Gamma)$ and the second by Proposition~\ref{prop:IRS}.
\end{proof}

\subsection{Unitary representations of charmenable and charfinite groups} \label{ss:unitary}

The fundamental fact that any positive definite function on a group could be viewed as a state on its universal C*-algebra is indispensable in this work.
Using it we get easily the following.

\begin{prop}
For a charfinite group, every unitary representation either weakly contains a finite dimensional subrepresentation
or weakly contains a representation which is induced from a finite normal subgroup.
\end{prop}


\begin{proof}
Let $\Gamma$ be a charfinite group and $\pi$ a unitary representation.
We let $C\subset \PD_1(\Gamma)$ be the $\Gamma$-invariant compact convex subset consisting of positive definite functions that extend to states on $C^*_\pi(\Gamma)$. By charmenability there exists a character $\phi\in C$ which we now fix. We get that the GNS representation $\pi_\phi$ is weakly contained in $\pi$.
If $\phi$ is supported on $\Rad(\Gamma)$ then $\pi_\phi$ is induced from $\Rad(\Gamma)$ which is finite.
Otherwise, $\pi_\phi$ is amenable thus it contains a finite dimensional subrepresentation, by Lemma \ref{soft charm}.
This subrepresentation is therefore weakly contained in $\pi$.
\end{proof}

Let $\Gamma$ be a discrete group. Any character of $\Gamma$ can be viewed as a trace on the universal C*-algebra $C^*(\Gamma)$. Naturally, for the regular character $\delta_e$, the corresponding trace on $C^*(\Gamma)$ is called the {\em regular trace}. Note that if $\pi$ is a unitary representation of $\Gamma$ then $\pi$ weakly contains the regular representation of $\Gamma$ if and only if the regular trace on $C^*(\Gamma)$ factors through the projection $C^*(\Gamma) \to C^*_\pi(\Gamma)$. We still call ``regular trace'' the trace obtained on $C^*_\pi(\Gamma)$ through this factorization.

\begin{prop}\label{prop:charmenable}
Assume $\Gamma$ is a charmenable discrete group with trivial amenable radical. 
Let $\pi$ be a unitary representation of $\Gamma$. Denote by $A := C^*_\pi(\Gamma)$. 
If $\pi$ is non-amenable then the following facts are true.
\begin{enumerate}
\item $\pi$ weakly contains the regular representation, the regular trace $\tau$ is the unique trace on $A$, and every proper ideal of $A$ is contained in $I_\tau := \{x \in A \mid \tau(x^*x) = 0\}$.
\item The regular trace $\tau$ on $A = C^*_\pi(\Gamma)$ satisfies a Powers property: for every $x \in A$, 
\[\tau(x) \in \overline{\conv}(\{\pi(g)x\pi(g)^* \mid g \in \Gamma\}).\]
\item There exists $\mu \in \Prob(\Gamma)$ whose support generates $\Gamma$ and such that the only $\mu$-stationary state on $A$ is $\tau$. 
\end{enumerate}
\end{prop}

\begin{proof}
(1) We let $C\subset \PD_1(\Gamma)$ be the $\Gamma$-invariant compact convex subset consisting of positive definite functions that extend to states on $A = C^*_\pi(\Gamma)$. By charmenability there exists a character $\phi\in C$ which we now fix. We get that the GNS representation $\pi_\phi$ is weakly contained in $\pi$.
Since $\pi$ is non-amenable we have that $\pi_\phi$ is non-amenable, and $\phi$ must be the regular character. Thus the GNS representation associated with $\phi$ is $\lambda$, and we conclude that $\lambda$ is weakly contained in $\pi$.

Further, if $I$ is a proper ideal in $A$, then represent faithfully $A/I$ on a Hilbert space $H$. The composed representation $\Gamma \to A \to A/I \to B(H)$ is still non-amenable, and thus must weakly contain the regular representation. So the canonical map $A \to C^*_\lambda(\Gamma)$, whose kernel is precisely $I_\tau$, factors through $A/I$. This shows that $I \subset I_\tau$.

(2) Arguing as in the previous point, every non-empty closed convex $\Gamma$-subset of $\cS(A)$ must contain a trace, which must be equal to $\tau$. 

{\bf Claim.} For every $n \geq 1$, every non-empty closed convex $\Gamma$-subset of $\cS(A)^n$ contains the fixed point $\tau^{(n)} := (\tau,\dots,\tau)$.

We prove this by induction. We just observed it was true for $n = 1$. Assume it is true for some $n \geq 1$ and take a non-empty closed convex $\Gamma$-subset $C \subset \cS(A)^{n+1}$. Then image of $C$ under the first projection map contains $\tau$ by the $n= 1$ step. Now the set $C \cap \{\tau\}\times \cS(A)^n$ is non-empty, and identifies with a closed convex $\Gamma$-subset of $\cS(A)^n$, which contains $\tau^{(n)}$ by our induction assumption. Thus $C$ contains $\tau^{(n+1)}$.

Using the above claim and a diagonal argument, we may find a sequence $(\mu_n)_{n \geq 1} \in (\Prob(\Gamma))^\N$ such that 
$(\mu_n \ast \phi)_n$ converges $*$-weakly to $\tau$ for every state $\phi \in \cS(A)$. In particular $\mu_n \ast x$ converges weakly to $\tau(x)1$ for every $x \in A$. Thus $\tau(x)1$ belongs to the weak closure of the convex hull of $\{\pi(g)x\pi(g)^* \mid g \in \Gamma\}$. By Hahn-Banach theorem, it belongs to its norm closure, which is the desired Powers property.

(3) This follows from (2) by the proof of \cite[Theorem 5.1]{HK17}.
Note that \cite[Theorem 5.1]{HK17} essentially states that (2) is equivalent to (3) for the regular representation $\lambda$, but its proof applies verbatim for an arbitrary unitary representation, showing that (2) implies (3).
\end{proof}

For $\mu \in \Prob(\Gamma)$, we say that $\psi \in \PD_1(\Gamma)$ is a $\mu$-{\em character} if $\psi$ is a $\mu$-stationary state on $C^*(\Gamma)$ with respect to the conjugation action. The following is an easy consequence.

\begin{prop}\label{prop:charmenable-C*}
Assume $\Gamma$ is a charmenable discrete group with trivial amenable radical. 
Then $\Gamma$ is C*-simple.
If in addition 
$\Gamma$ has property {\em (T)} then it is charfinite.
Further, in this case we may find $\mu \in \Prob(\Gamma)$ whose support generates $\Gamma$ such that every $\mu$-character on $\Gamma$ is a character.
\end{prop}

\begin{proof}
We assume as we may that $\Gamma$ is non-trivial, thus non-amenable.
The C*-simplicity of $\Gamma$ follows at once from Proposition~\ref{prop:charmenable}(1), applied to the regular representation
which is non-amenable.

Assume now $\Gamma$ has property (T)
and consider Definition~\ref{def:charmenable}.
By assumption, $\Gamma$ satisfies property (3).
It satisfies property (4) by  
\cite[Theorem~2.6]{Wa74}
and property (5) by
\cite[Theorem~1.1]{BV91}. 
It follows that $\Gamma$ is indeed charfinite.

Denote by $\pi$ the {\em universal weakly mixing} representation of $\Gamma$, i.e.\ the direct sum of all cyclic weakly mixing representations of $\Gamma$. Then since $\Gamma$ has property (T), $\pi$ is non-amenable and any representation weakly contained in $\pi$ is again weakly mixing.
Let $\mu \in \Prob(\Gamma)$ be a measure whose support generates $\Gamma$ and such that the only $\mu$-stationary state on $A  = C^*_\pi(\Gamma)$ is $\tau$, where $\tau$ is the regular trace on $A  = C^*_\pi(\Gamma)$, 
as guaranteed by Proposition~\ref{prop:charmenable}(3).

Let $\phi$ be a $\mu$-character.
We consider the compact convex subset of $\Gamma$-invariant positive definite functions which are dominated by $\phi$;
\[ S=\{\psi\in \PD(\Gamma)^\Gamma \mid \phi-\psi\in \PD(\Gamma)\} \subset \PD(\Gamma) \]
and claim that $\phi\in S$, thus $\phi$ is a character. 
We assume that this is not the case and argue to show a contradiction.
If $S=\{0\}$ we set $\phi_0=\phi$. Otherwise we let $\psi_0$ be a non-zero extreme point of $S$ and we set $\phi_0=(\phi-\psi_0)/(1-\psi_0(e))$.
In both cases, $\phi_0$ is a $\mu$-character which dominates no non-zero $\Gamma$-invariant positive definite function. 

{\bf Claim.} $\phi_0$ is a compact positive definite function.

By definition of $\pi$, a positive definite function on $\Gamma$ factorizes through a functional on $A = C^*_\pi(\Gamma)$ if and only if its GNS representation is weakly mixing, and conversely. Denote by $C \subset \PD(\Gamma)$ the closed convex subset consisting of such positive definite functions (including the null function). Then a positive definite function on $\Gamma$ is compact if and only if it does not dominate a positive definite function in $C$.
To prove the claim, we therefore consider the compact space 
\[ C_{0} = \{\psi\in C \mid \phi_0-\psi\in \PD(\Gamma)\} \subset C\]
and prove that $C_0 = \{0\}$. Note that this set is $\mu$-invariant. Take $\psi \in C_0$. By an averaging procedure, we may find $\psi_1 \in C_0$ which is $\mu$-stationary, and $\psi_1(1) = \psi(1)$. By definition of $C$, $\psi_1$ may be viewed as a multiple of a $\mu$-stationary state on $A$, and thus must be $\Gamma$-invariant, thanks to our choice of $\mu$. This forces $\psi_1 = 0$, since $\phi_0$ does not dominate non-zero invariant PD-functions. We thus find $\psi = 0$, and $C_0 = \{0\}$, proving our claim.

Since $\phi_0$ is compact, $M := \pi_{\phi_0}(\Gamma)\dpr$ is a tracial von Neumann algebra (contained in the direct sum of countably many finite dimensional algebras). Denote by $\Tr$ a normal faithful tracial state on $M$. Then $\Tr$ and $\phi_0$ are two $\mu$-stationary normal states on $M$. Since $\Tr$ is faithful and invariant, Proposition \ref{prop:stationary}(2) tells us that in fact, $\phi_0$ must be invariant as well. This is the desired contradiction.
\end{proof}

\subsection{Permanence properties} \label{ss:perm} 

\begin{prop}
Let $\Gamma$ be charmenable group.
Then for any normal subgroup $N\lhd \Gamma$, $\Gamma/N$ is charmenable.
Moreover, if $\Gamma$ is charfinite then so is $\Gamma/N$.
\end{prop}

\begin{proof}
The result is clear if $\Gamma/N$ is amenable. We thus may assume by Proposition~\ref{prop:NST} that $N$ is amenable,
thus $N\le \Rad(\Gamma)$.
We view $\PD_1(\Gamma/N)$ as a $\Gamma$-invariant closed subset of $\PD_1(\Gamma)$ in the obvious way.
The fact that every closed convex $\Gamma/N$-invariant subset of $\PD_1(\Gamma/N)$ contains a fixed point clearly follows from the corresponding property of $\Gamma$.
Note that $\Char(\Gamma/N)$ is in fact the subset of characters on $\Gamma$ that are equal to $1$ on $N$. Hence it is a face of $\Char(\Gamma)$.
in particular, extreme points of $\Char(\Gamma/N)$  are also extremal in $\Char(\Gamma)$.
So an extremal character on  $\Gamma/N$ which is not supported on $\Rad(\Gamma/N)$ may be viewed as an extremal character of $\Gamma$, which is not supported on $\Rad(\Gamma)$. In turn it must be von Neumann amenable, as a character of $\Gamma$, hence as a character of $\Gamma/N$.

The moreover part follows easily.
\end{proof}

The following proposition will be important for us when discussing charmenability of lattices in infinite restricted products.

\begin{prop}\label{prop:stable}
Let $\Gamma = \bigcup_{n \in \N} \Gamma_n$ be an ascending union.
If for every $n \in \N$, $\Gamma_n$ is charmenable, then so is $\Gamma$.
\end{prop}
\begin{proof}
Let $C\subset \PD(\Gamma)$ be a compact convex $\Gamma$-invariant subset.
For every $n$, let $C_n$ be the preimage of the $\Gamma_n$-invariants in the image of the restriction map $\PD(\Gamma)\to \PD(\Gamma_n)$.
Then $(C_n)_n$ is a descending sequence of compact sets, hence has non-trivial intersection. This shows that the set of $\Gamma$-invariants in $C$ is non-empty.
Fix an extremal character $\phi\in \Char(\Gamma)$ and assume the support of $\phi$ is not contained in $\Rad(\Gamma)$.
We argue to show that $\phi$ is von Neumann amenable.
We note that 
\[ \Rad(\Gamma)= \bigcup_{k=1}^{\infty} \bigcap_{n=k}^\infty \Rad(\Gamma_n). \]
Indeed, the inclusion $\subset$ is clear and $\supset$ follows from the fact that $g\in \Rad(\Gamma)$ iff for every finite set $F\subset \Gamma$,
the group generated by $\{g^f\mid f\in F\}$ is amenable.
We fix $g_0\in \Gamma\setminus\Rad(\Gamma)$ such that $\phi(g_0)\neq 0$.
By passing to a subsequence we assume as we may that 
for every $n$, $g_0 \in \Gamma_n \setminus \Rad(\Gamma_n)$.
For every $n \in \N$, we let $\phi_n=\phi_{|\Gamma_n} \in \Char(\Gamma_n)$.

{\bf Claim.} For every $n$, $\phi_n$ is a von Neumann amenable character of $\Gamma_n$.

Fix an index $n$, and denote by $N_n$ the GNS von Neumann algebra associated with $(\Gamma_n,\phi_n)$, and denote by $\tau$ the normal trace on $N_n$ extending $\phi_n$. We want to prove that $N_n$ is an amenable von Neumann algebra.
Apply Lemma~\ref{soft charm} to $\phi_m$, $m \geq n$, to get a decomposition 
\[ \phi_m = t_m \phi_m^1 + (1-t_m)\phi_m^2,\]
for some $t_m \in [0,1]$ and characters $\phi_m^1,\phi_m^2\in \Char(\Gamma_m)$ such that $\phi_m^1$ is von Neumann amenable as a character of $\Gamma_m$, and $\phi_m^2$ is supported on $\Rad(\Gamma_m)$.

Note that such a decomposition may be restricted to $\Gamma_n$, and that the restriction $\phi_m^1$ to $\Gamma_n$ is still von Neumann amenable. By Lemma \ref{convexity of amenable reps}, we may find a projection $p \in N_n$, with trace at least $t_m$ such that $pN_np$ is an amenable von Neumann algebra. So the claim will follow once we prove that $t_m$ tends to $1$ as $m$ goes to infinity. This later fact relies on the extremality assumption.

For every $m \geq n$, denote by $\psi_m^1$ and $\psi_m^2$ the positive definite extensions to $\Gamma$ of $\phi_m^1$ and $\phi_m^2$, respectively, obtained by assigning the value $0$ outside of $\Gamma_m$. Passing to a subsequence if necessary, we assume as we may that the sequences $\psi_m^1,\psi_m^2$ and $t_m$ all converge
and we observe that the limit functions are characters of $\Gamma$.
Since the sequence $t_m \psi_m^1 + (1-t_m)\psi_m^2$ converges to $\phi$ which is an extremal character, since $\phi(g_0)\neq 0$ while for every $m$, $\psi_m^2(g_0)=0$, it follows that $t_m \to 1$, as desired. This concludes the proof of the claim.

Let us now deduce that $\phi$ is a von Neumann amenable character on $\Gamma$. Denote by $(H,\pi,\xi)$ the GNS triple associated with $\phi$, and for every $n$, denote by $(H_n,\pi_n,\xi_n)$, the GNS triple of $\phi_n$. Naturally $H_n$ can be viewed as a subspace of $H$, in such a way that $\xi_n$ coincides with $\xi$, and $\pi_n$ is the restriction of $\pi$ to $H_n$. Since $\Gamma = \bigcup_n \Gamma_n$, we find that the increasing union of the spaces $H_n$ is dense in $H$. 

Define $M = \pi(\Gamma)\dpr$ and $M_n := \pi(\Gamma_n)\dpr$, for each index $n$. Denoting by $p_n: H \to H_n$ the orthogonal projection, we find that $p_n \in M_n'$ and $p_nM_n \simeq \pi_n(\Gamma_n)\dpr$ is amenable for all $n$. In particular $p_mM_n \subset p_m M_m$ is also amenable for all $m \geq n$. Now $p_m$ converges strongly to $1$, so we find that $M_n$ is amenable, and $M = (\bigcup_n M_n)\dpr$ follows amenable as well.
\end{proof}

\section{$(G,N)$-structures, singularity and criteria for charmenability}\label{sect:criterion}

Throughout this section $G$ denotes an lcsc group.
The first three subsections will be devoted to the study of $(G,N)$-von Neumann algebras with a focus on their singularity properties.
This study will be used to develop charmenability criteria in \S\ref{ss:charcrit}.


\subsection{Definition and examples of $(G,N)$-structures}\label{ss:GN}

The setting of $\mu$-stationary actions is quite general, but it is sometimes too loose for our purposes. This is because Furstenberg-Poisson boundaries don't always behave well when passing to subgroups, even to lattices. In contrast the amenability of an action remains when restricting to any closed subgroup.

For this reason, we want to study the more general data of a $G$-action $G \actson A$ on a C*-algebra $A$ together with a $G$-map $\theta: B \to \cS(A)$ from an amenable $G$-space $(B,\nu)$. In the commutative case, such boundary maps $\theta: B \to \Prob(X)$ naturally give rise to a measure class $\bary(\theta_*\nu)$ on $X$. Then measurable notions on $X$, such as ergodicity, are discussed. 

In the non-commutative case, it is the same: a boundary map $\theta: (B,\nu) \to \cS(A)$ naturally comes with a state $\phi = \bary(\theta_*\nu)$ and we want to study ``measurable'' aspects of the GNS von Neumann algebra $M = \pi_\phi(A)\dpr$ (such as ergodicity). In fact, we can keep track of $\theta$ purely in terms of $M$, as follows. By duality, the map $\theta$ gives rise to a $G$-ucp map $E: A \to L^\infty(B)$. If we denote by $z \in A^{**}$ the central support projection of the normal extension $E : A^{**} \to L^\infty(B)$, then $z$ is $G$-invariant (with respect to the normal extension of the action), and $M$ is naturally isomorphic with $zA^{**}$. Moreover Proposition \ref{continuous vn extension} shows that $M$ is indeed a $G$-von Neumann algebra. The map $E$ can thus be viewed as a normal $G$-ucp map $M \to L^\infty(B)$.

Note that in this picture, if $A$ is separable, we recover the initial map $\theta$ form the composition $A \to M \to L^\infty(B)$ by duality. So the two points of view are equivalent, but the advantage of expressing things in terms of $M$ is that this will allow us to change the compact model of the action.

\begin{defn}\label{defn:GN-vN}
Let $N$ be a $G$-von Neumann algebra. A {\em $(G,N)$-von Neumann algebra} will be the data $(M,E)$ of a $G$-von Neumann algebra $M$ and a normal $G$-ucp map $E : M \to N$. We will sometimes refer to $E$ as the $(G,N)$-structure map. We say that $(M, E)$ is {\em faithful} or {\em extremal} if $E$ 
satisfies the corresponding properties. If $E(M) \subset N^G$, we say that $(M, E)$ is $G$-{\em invariant}.
\end{defn}

Classically, stationary states give rise to boundary maps. This is our first example.

\begin{prop}\label{stationary structure}
Fix a generating probability measure $\mu \in \Prob(G)$, denote by $(B,\nu)$ the corresponding Furstenberg-Poisson boundary and set $N = L^\infty(B,\nu)$. We view $\nu$ as the state on $N$ given by integration w.r.t.\ the measure $\nu$. Let $M$ be a $G$-von Neumann algebra.

Then a $(G,N)$-structure map $E: M \to N$ gives rise to a $\mu$-stationary normal state $\varphi = \nu \circ E$ on $M$. Conversely, a normal $\mu$-stationary state $\varphi$ on $M$ gives rise to a normal $G$-ucp map $E: M \to \Har_\mu(G) \simeq L^\infty(B)$, defined by $E(x)(g) = \varphi(g^{-1}(x))$, for all $x \in M$, $g \in G$. These two maps are inverse of one another. Moreover, 
\begin{itemize}
\item $E$ is faithful if and only if $\varphi$ is faithful;
\item $E$ is extremal if and only if $\varphi$ is extremal;
\item $E$ is invariant if and only $\varphi$ is $G$-invariant.
\end{itemize}
\end{prop}
\begin{proof}
The fact that the two maps $E \mapsto \varphi$ and $\varphi \mapsto E$ are inverse of each other follows from the definition of the Poisson transform $\Har_\nu(G) \simeq L^\infty(B,\nu)$ (see e.g.\ \cite[\S 2]{BS04}). The other statements are trivial.
\end{proof}

Let us give now a purely non-commutative example.

\begin{exmp}\label{character structure}
Assume that $G$ is discrete and take an arbitrary $G$-algebra $N$, whose action is denoted by $\sigma$. Let $M$ be any tracial factor with separable predual and $\pi : G \to \mathcal U(M)$ any unitary representation such that $\pi(G)\dpr = M$. In other words, $M$ is the GNS von Neumann algebra of an extremal character on $G$. We denote by $(L^2(M),L^2(M)_+,J)$ the standard form of $M$.

The group $G$ acts on $N \ovt B(L^2(M))$ via the automorphisms $\sigma_g \otimes \Ad(J \pi(g) J)$, $g \in G$. The fixed point algebra $\cM = (N \ovt B(L^2(M)))^G$ admits another action $\tsigma: G \curvearrowright \cM$, given by $\Ad(1_N \otimes \pi(g))$, for $g \in G$. 

Denote by $\xi \in L^2(M)_+$ the unique cyclic vector implementing the trace $\tau$ and by $\Phi = \langle \, \cdot \, \xi, \xi \rangle \in B(L^2(M))_\ast$ the corresponding normal vector state. Since $\tau$ is a trace, we know that $a\xi = \xi a = Ja^*J\xi$ for every $a \in M$. Then for every $g \in G$, the normal ucp map $E = \id_N \ot  \Phi : \cM \to N$ satisfies 
\begin{align*}
E \circ \tsigma_g = \id_N \otimes (\Phi \circ \Ad(\pi(g)))&=  \id_N \otimes (\Phi \circ \Ad(J\pi(g)^*J)) \\
&= (\id_N \otimes \Phi) \circ (\Ad(1_N \otimes J\pi(g)^*J)) \\
&= (\id_N \otimes \Phi)  \circ ( \sigma_g \otimes \id) \\
&= \sigma_g \circ E.
\end{align*}
We used the invariance property for elements in $\cM$ to obtain the third line above. This shows that $(\mathcal M, E)$ is a $(G, N)$-von Neumann algebra.
\end{exmp}

\begin{lem}\label{faithful character structure}
In the above example, the structure map $E$ is faithful.
\end{lem}
\begin{proof}
We keep the notation from the previous example. Let $x \in \cM$ be such that $E(x^*x) = 0$. Then for every $g \in G$, we get $E((1 \ot \pi(g)^*)x^*x(1 \ot \pi(g))) = \sigma_g^{-1}(E(x^*x)) = 0$. 

Now, viewed as a ucp map on $N \ovt B(L^2(M))$, the support of $E = \id_N \ot \Phi$ is $1 \ot p_\xi$, where $p_\xi$ is the rank one projection onto $\C \xi$. We thus find that $x(1 \ot \pi(g)p_\xi) = 0$ for every $g \in G$. Since $\pi(g)\xi$, $g \in G$, spans a dense subspace of $L^2(M)$, this indeed implies $x = 0$.
\end{proof}

In our last example we explain how to induce structure maps from a lattice to the ambient group $G$.

\begin{exmp}\label{induced structure}
Let $\Gamma < G$ be a lattice, $N$ be a $\Gamma$-von Neumann algebra and $(M,E)$ any $(\Gamma, N)$-von Neumann algebra. Equally denote by $\sigma$ the $\Gamma$ actions on $M$ and $N$. Denote by $\lambda$ and $\rho$ the translation actions of $G$ on $L^\infty(G)$ on the left and right, respectively. Define the fixed point von Neumann algebras 
\[\tM := (L^\infty(G) \ovt M)^{(\rho \ot \sigma)(\Gamma)} \quad \text{and} \quad \tN := (L^\infty(G) \ovt N)^{(\rho \ot \sigma)(\Gamma)}.\]
Since $E$ is $\Gamma$-equivariant, the map $\tE = \id \ot E$ maps $\tM$ into $\tN$. Moreover, this map clearly intertwines the induced $G$-actions $\lambda \ot \id$ on $\tM$ and $\tN$.
We call $(\tM,\tE)$ the {\em induced} $(G, \tN)$-von Neumann algebra. Note that it is faithful if $E$ is.

In the special case where $N$ is already a $G$-von Neumann algebra, we further have a faithful $G$-equivariant normal ucp map $E_N: \tN \to N$. Indeed, in this case the induced $G$-action on $\tN$ is conjugate with the diagonal $G$-action on $L^\infty(G/\Gamma) \ovt N$. Then $E_N$ is given by integrating on the first component, $E_N = m_{G/\Gamma} \ot \id$. So in this case we also get a $(G,N)$ structure $E_N \circ \tE$ on~$\tM$.
\end{exmp}

\begin{lem}\label{invariant induced}
Keep the setting of the above example, and assume that $N$ is a $G$-algebra on which $\Gamma$ acts ergodically\footnote{In fact the assumption that $N^\Gamma = N^G$ would also imply the equivalence between $(1)$ and $(3)$.}. The following are equivalent:
\begin{enumerate}
\item $E$ is $\Gamma$-invariant.
\item $\tE$ ranges into $(L^\infty(G) \ot 1)^{(\rho \ot \sigma)(\Gamma)} \simeq L^\infty(G/\Gamma)$;
\item $E_N \circ \tE$ is $G$-invariant.
\end{enumerate}
\end{lem}
\begin{proof}
The implications $(1) \Rightarrow (2) \Rightarrow (3)$ are clear thanks to our ergodicity assumption.

$(2) \Rightarrow (1)$. Fix $x \in M$. We want to prove that $E(x)$ is a scalar operator. Choose a fundamental domain $\cF \subset G$ for the right $\Gamma$-action on $G$. Let $f \in \tM$ be the $\Gamma$-equivariant function $G \to M$ which is equal to $\sigma_\gamma^{-1}(x)$ on the translate $\cF \gamma$ of $\cF$, for all $\gamma \in \Gamma$. Then $\tE(f) \in \tN$ is constant on $\cF$, equal to $E(x)$. By $(2)$, this function has scalar values, so $E(x) \in \C1$.

$(3) \Rightarrow (2)$. Note that $L^\infty(G/\Gamma)$ may be viewed as a subalgebra of both $\tM$ and $\tN$, and that it lies in the multiplicative domain of $\tE$. Hence the image of $\tE$ is an $L^\infty(G/\Gamma)$-module.

Let us describe explicitly the map $E_N: \tN \to N$. Identify $\tN$ with the algebra of $\Gamma$-equivariant functions from $G$ to $N$ (with respect to the right action of $G$ on itself). For $f \in \tN$, define $\theta(f): G \to N$ by the formula $\theta(f)(g) = \sigma_g(f(g))$, for every $g \in G$. Since $f$ is $\Gamma$-equivariant, $\theta(f)$ is right $\Gamma$-invariant and thus defines an $N$-valued function on $G/\Gamma$, that is, an element of $L^\infty(G/\Gamma) \ovt N$. The map $\theta: \tN \to L^\infty(G/\Gamma) \ovt N$ defined this way is an onto $*$-isomorphism, which intertwines the induced action on $\tN$ with the diagonal $G$-action on $L^\infty(G/\Gamma) \ovt N$. Then 
\[E_N = (m_{G/\Gamma} \ot \id) \circ \theta: \tN \to N.\]
Observe that $\theta$ maps the algebra $(L^\infty(G) \ot 1)^{(\rho \ot \sigma)(\Gamma)} \subset \tN$ onto the subalgebra $L^\infty(G/\Gamma) \ot 1$. 
Take $f \in \tN$ in the image of $\tE$. By $(3)$, the element $\theta(f) \in L^\infty(G/\Gamma) \ovt N$ is such that 
\[(m_{G/\Gamma} \ot \id)(a\theta(f)) \in \C 1, \text{ for every } a \in L^\infty(G/\Gamma).\]
This implies that the essential range of $\theta(f)$, viewed as an $N$-valued function on $G/\Gamma$, is contained in $\C 1$. Hence $\theta(f) \in L^\infty(G/\Gamma) \ot 1$, as desired.
\end{proof}

\begin{rem}\label{rem:induction}
Let $\Gamma$ be a lattice in $G$, and $\mu \in \Prob(G)$ be a generating measure. Assume that $(B, \nu)$ is the $(G,\mu)$-Furstenberg-Poisson boundary and set $N = L^\infty(B,\nu)$. Assume moreover that $(B, \nu)$ is the $(\Gamma,\mu_0)$-Furstenberg-Poisson boundary for some admissible probability measure $\mu_0 \in \Prob(\Gamma)$. 
Let $M$ be a $\Gamma$-von Neumann algebra with a $\mu_0$-stationary state $\varphi$. By Example \ref{stationary structure}, there is a unique normal ucp $\Gamma$-map $E : M \to N$ so that $\nu \circ E = \varphi$. Using the previous observation, the normal state $\varphi = \nu \circ E_N \circ \tE$ on $\tM$ is faithful and $\mu$-stationary. This provides a more conceptual view of \cite[Theorem 4.3]{BH19}.
\end{rem}

\subsection{Singular states and singular structures} \label{ss:singstate}


\begin{defn}
Two states $\phi,\psi$ on a C*-algebra $A$ are called {\em singular}, denoted by $\phi \perp \psi$, if $\Vert \phi - \psi\Vert = 2$.
\end{defn}

\begin{lem}
Consider two states $\phi$, $\psi$ on the unital C*-algebra $A$. The following are equivalent:
\begin{enumerate}
\item $\phi \perp \psi$.
\item The support projections\footnote{We emphasize that we are talking here about the support projections and not the central support projections.} of $\phi$ and $\psi$ in $A^{**}$ are perpendicular;
\item There exists a sequence $(a_n)_n \in A^{\N}$ such that $0 \leq a_n \leq 1$, $\lim_n \phi(a_n) = 0$ and  $\lim_n \psi(a_n) = 1$.
\end{enumerate}
\end{lem}
\begin{proof}
$(1) \Rightarrow (3)$. By assumption, we may find a sequence $(x_n)_{n \in \N}$ in the unit ball of $A$ such that $\lim_n (\phi(x_n) - \psi(x_n)) = 2$. Replacing $x_n$ by its real part $(x_n + x_n^*)/2$ if necessary, we may assume that $x_n$ is self-adjoint for every $n$, hence $-1 \leq x_n \leq 1$. Since $\phi(x_n) \leq 1$ and $\psi(x_n) \geq -1$ for every $n$, we must actually have $\lim_n \phi(x_n) = 1$ and $\lim_n \psi(x_n) = -1$. So the elements $a_n := (1 - x_n)/2$ satisfy the conclusion of $(3)$.

$(3) \Rightarrow (2)$. Extend $\phi$ and $\psi$ to normal states on $A^{**}$. By definition, the support of $\phi$ is the smallest projection $r \in A^{**}$ such that $\phi(r) = 1$. Denote by $x \in A^{**}$ an ultraweak limit of the sequence $(a_n)_n$ given by $(3)$. Then $0 \leq x \leq 1$, $\phi(x) = 0$ and $\psi(x) = 1$. Denote by $p \in A^{**}$ the support projection of $x$. Then we have $\phi(p) = 0$ and $\psi(p) = 1$. So $p$ dominates the support of $\psi$ and $1-p$ dominates the support of $\phi$. This proves $(2)$.

$(2) \Rightarrow (1)$. Note that the norm of $\phi - \psi$ is equal to the norm of its normal extension to $A^{**}$. Assuming $(2)$, the support of $\psi$ is a projection $p \in A^{**}$ such that $\phi(p) = 0$, $\psi(p) = 1$. Then the element $x = 1-2p \in A^{**}$ satisfies $-1 \leq x \leq 1$ and $(\phi - \psi)(x) = 2$.
 \end{proof}

Observe that singularity passes to larger algebras: if $A \subset B$ is an inclusion of C*-algebras and $\phi, \psi \in \cS(B)$ have their restrictions to $A$ that are singular, then $\phi$ and $\psi$ themselves are singular. 

The following proposition extends this definition to boundary maps, and proves independence on the choice of a compact model.

\begin{prop}\label{singular maps}
Consider a separable von Neumann algebra $M$ and a probability space $(B,\nu)$ and two normal ucp maps $E_1,E_2: M \to L^\infty(B)$. For every separable C$^*$-subalgebra $A \subset M$, the restriction maps $E_k: A \to L^\infty(B)$, dualize to measurable maps $\theta_k^A: B \to \cS(A)$ for all $k \in \{1, 2\}$. The following are equivalent.
\begin{enumerate}
\item For every weakly dense, separable unital C*-subalgebra $A \subset M$, we have $\theta_1^{A}(b) \perp \theta_2^{A}(b)$, for almost every $b \in B$. 
\item There exists a separable unital C*-algebra $A \subset M$ such that $\theta_1^{A}(b) \perp \theta_2^{A}(b)$, for almost every $b \in B$. 
\item For every $\eps > 0$, there exist finitely many projections $p_1,\dots,p_N \in L^\infty(B)$, such that $\sum_{n=1}^N p_n = 1$ and $\Vert \phi_{1,n} - \phi_{2,n}\Vert \geq 2 - \eps$, for all $1 \leq n \leq N$, where $\phi_{1,n}, \phi_{2,n} \in M_*$ are the normal states defined by
\[\phi_{k,n}:  x \in M \mapsto \frac{1}{\nu(p_n)}\nu(p_nE_k(x)), \text{ for  all } k \in \{1,2\}, 1 \leq n \leq N.\]
\end{enumerate}
\end{prop}

\begin{defn}
Two normal ucp maps $E_1$, $E_2$ satisfying the above equivalent conditions are called {\em singular}.
\end{defn}

\begin{rem} \label{rem:subsing}
Note that by (2), if $M_0\le M$ is a von Neumann subalgebra and  $E_1,E_2: M \to L^\infty(B)$ are normal ucp maps whose restrictions to $M_0$ 
are singular normal ucp maps then $E_1$ and $E_2$ are singular normal ucp maps.
\end{rem}

\begin{proof}[Proof of Proposition \ref{singular maps}]
$(1) \Rightarrow (2)$ follows from the separability of $M$.

$(2) \Rightarrow (3)$. Assume that $(2)$ is true and take $\eps > 0$. Observe that statement $(3)$ has some flexibility. Instead of looking for a true partition of unity we may only look for pairwise orthogonal projections $p_1,\dots,p_n$ such that $\nu(1 - \sum_i p_i) < \eps$ and which satisfy the conclusion of $(3)$. Indeed, once this is achieved, we may distribute the remaining mass (of size $< \eps$) proportionally to each $p_i$, to form new projections $q_i$ which actually add up to $1$, and still satisfy the rest of the conclusion (up to inflating $\eps$ in a non-essential way).

Since $A$ is separable, we may find a sequence $(x_n)_{n \in \N}$ in $A$, which is dense in $A_{[0,1]} := \{x \in A \mid 0 \leq x \leq 1\}$. For each $n \in \N$, we define 
\[B_n^0:= \left\{b \in B \mid \theta_1^A(b)(x_n) \leq \eps/4 \quad \text{and} \quad \theta_2^A(b)(x_n) \geq 1 - \eps/4\right \}.\]
By density of the sequence $(x_n)_n$ in $A_{[0,1]}$, condition $(2)$ tells us that $\bigcup_{n\in \N} B_n^0$ is co-null in $B$. Let us pick pairwise disjoints sets $(B_n)_{n \in \N}$ of $B$ such that $B_n \subset B_n^{0}$ for every $n \in \N$, and still $\bigcup_n B_n$ is co-null in $B$. Fix $N$ large enough so that $\nu(1 - \bigcup_{n=1}^N B_n) < \eps$.

Then the projections $p_i = \mathbf{1}_{B_i}$, $i = 1 \dots N$, satisfy the desired conclusion. Indeed, for every $i \leq N$, we have $\Vert 1 - 2x_i \Vert \leq 1$, while
\[p_iE_1(1 - 2x_i) - p_iE_2(1 - 2x_i) \geq p_i(1 - 2\eps/4) - p_i(1 - 2(1-\eps/4)) = (2 - \eps)p_i.\]

$(3) \Rightarrow (1)$. Fix a separable dense C*-subalgebra $A \subset M$, and $\eps > 0$. We claim that the set 
\[B_\eps := \{b \in B \mid \Vert \theta_1^A(b) - \theta_2^A(b) \Vert \geq 2 - \eps\}\]
has measure at least $1 - \eps$. This claim clearly implies $(1)$.

Take projections $p_1,\dots, p_N \in L^\infty(B)$ as in condition $(3)$, with respect to some $\delta < \eps^2$.
Note that the corresponding states $\phi_{1,n}$, $\phi_{2,n}$, $n = 1, \dots, N$ are normal on $M$. So by Kaplansky density theorem, we have
\[\Vert \phi_{1,n} - \phi_{2,n}\Vert = \sup_{x \in A_{[-1,1]}} (\phi_{1,n}(x) - \phi_{2,n}(x)).\]
So for every $n = 1, \dots, N$, we may find a self-adjoint element $x_n \in A$, $\Vert x_n \Vert \leq 1$, such that $ \phi_{1,n}(x_n) - \phi_{2,n}(x_n) \geq 2 - \eps^2$.
Fix $n \leq N$,  and define $p_n := \mathbf{1}_{B_n}$, $f_n = p_n(E_1(x_n) - E_2(x_n)) \in L^\infty(B)$ and
\[B_{n,\eps} := \{ b \in B_n \mid  (2 - f_n)(b) \leq \eps\}.\]
Note that $2 - f_n$ has non-negative real values. By Markov inequality, we have
\begin{align*}
\nu(B_n \setminus B_{n,\eps}) \eps & \leq \int_{B_n} (2-f_n)(b) \dd\nu(b)\\
& = 2\nu(B_n) - \nu(B_n) ( \phi_{1,n}(x_n) - \phi_{2,n}(x_n) )\\
& \leq \nu(B_n)\eps^2.
\end{align*}
So we find $\nu(B_{n,\eps}) \geq (1 - \eps)\nu(B_n)$. Observe that $B_{n,\eps} \subset B_n \cap B_\eps$. So adding up over $n$, we get $\nu(B_\eps) \geq 1 - \eps$, as desired.
\end{proof}

\begin{defn} \label{def:Eg}
Let $G$ be an lcsc group. Let $(B,\nu)$ be a $G$-space and set $N = L^\infty(B)$. Let $(M,E)$ be a separable $(G,N)$-von Neumann algebra. We say that $(M,E)$ is a {\em singular $(G,N)$-von Neumann algebra} if the normal ucp maps $E_g: x \in M \mapsto E(gx) \in L^\infty(B)$, $g \in G$, are pairwise singular.
\end{defn}

\subsection{Singularity criterion} \label{ss:singcrit}

Our next goal is to give a criterion for singularity of $(G,N)$-algebras for some $G$ and $N$.
It is inspired by \cite[Section 2, Theorem 2.5]{BF14}.

\begin{lem}
Let $G$ be an lcsc group and $P\le G$ a closed subgroup which is stably self normalizing and has the relative Mautner property in $G$
as defined in Definition~\ref{def:RMP}.
Let $A$ be a separable $G$-C*-algebra and $\phi$ an extremal $P$-invariant state on $A$. 

For every $g \in G$, either $g\phi = \phi$ or $g\phi \perp \phi$.
\end{lem}
\begin{proof}
We fix $g\in G$ such that $\phi$ and $g\phi$ are not singular and argue to show that $\phi$ is $g$-fixed.
We denote by $H$ the stabilizer of $\phi$ in $G$. It is a closed subgroup of $G$ containing $P$ so we only need to check that $g$ normalizes $H$. This amounts to show that $g\phi$ and $g^{-1}\phi$ are both $H$-invariant. 

Extend the $H$-action on $A$ to a non-continuous action on $A^{**}$. We still denote by $\phi$ the normal extension of $\phi$ to $A^{**}$ and we denote by $z \in \mathcal Z(A^{**})$ its central support projection. 
Note that $z$ is $H$-invariant since $\phi$ is $H$-invariant.
By Proposition~\ref{continuous vn extension}, $H$ acts continuously on $M := zA^{**}$, which means that the isometric action on the Banach space $M_*$ is continuous. The Mautner property will be applied to this specific isometric action.

We consider the central projection $\sigma_g(z) \in \mathcal Z(A^{**})$ and the normal positive linear functional 
\[ \phi_g: a \in M = zA^{**} \mapsto \phi(\sigma_g(z)a)=\phi(z\sigma_g(z)a).  \]
We observe that $\phi_g \in M_*$ is $P \cap gPg^{-1}$-invariant. By assumption, the pair $(P,P\cap gPg^{-1})$ has the relative Mautner property so $\phi_g$ also $P$-invariant.
Clearly, $\phi_g \leq \phi$ so by extremality of $\phi$ as a $P$-invariant state, $\phi_g$ must be proportional to $\phi$. In terms of the central support projection, this tells us that $z\sigma_g(z)$ is either null (in case $\phi_g = 0$) or it is equal to $z$.
We assumed that $\phi$ and $g\phi$ are not singular, so $z\sigma_g(z) = 0$ is excluded and we get $z\sigma_g(z)=z$.
We conclude that $z \leq \sigma_g(z)$.

Considering similarly the functional $\phi_{g^{-1}}$ and using the fact that $g^{-1}\phi$ and $\phi$ are not singular, we also get $z \le \sigma_{g^{-1}}(z)$, thus $\sigma_g(z) \leq z$.
We conclude that $z = \sigma_g(z)$. Since $\sigma_g(z)$ is the central support of $g\phi$ this means that $g\phi$ is in fact a normal state on $M = zA^{**}$. By assumption, this state is $P \cap gPg^{-1}$-invariant. By Lemma~\ref{lem:RMP} the pair $(H,P \cap gPg^{-1})$ has the Mautner property. Considering again the continuous isometric action of $H$ on $M_*$, we conclude that $g\phi$ is $H$-invariant. The same argument shows that $g^{-1}\phi$ is $H$-invariant, concluding the proof.
\end{proof}

\begin{prop}[Singularity criterion] \label{FMM implies singular}
Let $G$ be an lcsc group and $P\le G$ a closed subgroup. Assume that $P$ is stably self normalizing and that it has the relative Mautner property in $G$.
Denote by $N := L^\infty(G/P,\nu)$, where $\nu$ is a $G$-quasi-invariant Radon measure. 

Let $(M,E)$ be an extremal separable $(G,N)$-von Neumann algebra which is not $G$-invariant.
If $g \in G$ has a null set of fixed points in $G/Q$ for any intermediate proper closed subgroup $P < Q < G$, then $E$ and $E_g$ are singular.

In particular, if $G$ acts essentially freely on $G/Q$ for any intermediate proper closed subgroup $P < Q < G$, then $(M,E)$ is a singular $(G,N)$-von Neumann algebra. 

\end{prop}
\begin{proof}
The extremality assumption means that $E$ is extremal among the (normal) $G$-ucp maps $M \to N$. In particular, it implies that the restriction of $E$ to any weakly dense $G$-C*-subalgebra $A$ is extremal. Choose such a C*-subalgebra $A \subset M$, and assume that $A$ is separable. Then the restriction of $E$ to $A$ corresponds to a $G$-equivariant map $\theta: G/P \to \cS(A)$. Since $G$ acts transitively on $G/P$, we may assume that $\theta$ is everywhere defined and everywhere equivariant. In other words, there exists a $P$-invariant state $\phi$ on $A$ such that $\theta(gP) = g\phi$ for every $gP \in G/P$. 

The extremality condition implies that $\phi$ is extremal among $P$-invariant states on $A$. So by the previous lemma, we find that for every $g \in G$, either $g\phi = \phi$ or $g\phi \perp \phi$.

Assume that $E$ is not $G$-invariant, which implies that $\phi$ is not $G$-invariant. Denote by $Q < G$ the stabilizer of $\phi$. This is a proper closed subgroup and for every $g,h \in G$, we have $g\theta(hP) \perp \theta(hP)$ as soon as $ghQ \neq hQ$, or equivalently, as soon as $hQ$ is not in the fixed point set of $g$ inside $G/Q$. 
If $g \in G$ has a null set of fixed points in $G/Q$, then for almost every $hP \in G/P$, $g\theta(hP) \perp \theta(hP)$. This is exactly the singularity condition $E \perp E_g$.
\end{proof}

In order to apply the above proposition, we will need to verify the extremality condition on maps $E$. The following lemma gives a sufficient condition. 

\begin{lem}\label{ergodic extremal}
Take an lcsc group $G$ with a generating measure $\mu \in \Prob(G)$ and denote by $(B,\nu)$ the corresponding Furstenberg-Poisson boundary and by $N = L^\infty(B,\nu)$.

If $M$ is an ergodic $G$-von Neumann algebra, then there exists at most one $(G,N)$-structure map $E: M \to N$. In particular it is extremal.
\end{lem}
\begin{proof}
By Example \ref{stationary structure}, a structure map $E$ is the same data as the normal $\mu$-stationary state $\phi = \nu \circ E$. But we saw in Proposition \ref{prop:stationary} that if $M$ is ergodic, there exists at most one normal $\mu$-stationary state on $M$.
\end{proof}

\subsection{Charmenability criteria} \label{ss:charcrit}

The proof of our main results will rely on the following criterion. It is an abstract version of the techniques used in \cite{BH19}.
Recall Definition~\ref{def:ME} of metric ergodicity.

\begin{prop}\label{top crit}
Let $\Gamma$ be a discrete group with trivial amenable radical. Assume that there exist a $\Gamma$-space $(B,\nu)$ which is separable, amenable and metrically ergodic such that $N = L^\infty(B)$ satisfies the following property. \begin{description}
\item[(a)] Every separable, ergodic, faithful $(\Gamma,N)$-von Neumann algebra $(M,E)$ is either invariant or $\Gamma$-singular.
\end{description}
Then $\Gamma$ is charmenable.
\end{prop}
\begin{proof}
We have two statements to verify.

{\em Fixed point property.} Let $C \subset \PD_1(\Gamma)$ be a closed convex $\Gamma$-subset. We need to show that $C$ contains a character $\phi$, i.e. a $\Gamma$-fixed point. Denote by $A = C^*(\Gamma)$ the universal C*-algebra of $\Gamma$, endowed with the conjugacy $\Gamma$-action by the unitaries $u_g$, $g \in \Gamma$. We may view $C$ as a compact convex $\Gamma$-subset of $\cS(A)$. 

By amenability of $(B,\nu)$, we may find a measurable $\Gamma$-map $\theta: B \to C$. We claim that the state $\phi := \bary({\theta}_*\nu) \in C$ is $\Gamma$-invariant. In fact, we claim that this holds for every measurable $\Gamma$-map $\theta: B \to \cS(A)$.

The set $\tC$ of such maps $\theta$ is a convex set, and since $A$ is separable, there is an affine bijection between $\tC$ and the convex set of $\Gamma$-equivariant ucp maps $E: A \to L^\infty(B)$. When endowed with the topology of pointwise ultraweak convergence, this later convex set of ucp maps is compact. This induces on $\tC$ a structure of compact convex space. So by Krein-Milman it is the closed convex hull of its extremal points. Moreover, the map $\theta \in \tC \mapsto \bary(\theta_*\nu) \in \cS(A)$ is affine and continuous. So it suffices to prove our claim under the assumption that $\theta$ is an extremal map in $\tC$.

Let $\theta: B \to \cS(A)$ be an extremal map, and denote by $E: A \to L^\infty(B)$ the corresponding $\Gamma$-equivariant ucp map. 
We may extend the $\Gamma$-action on $A$ to an action on $A^{**}$, and we may also extend $E$ to a normal $\Gamma$-ucp map $A^{**} \to L^\infty(B)$. 
We denote by $z \in \cZ(A^{**})$ the central support of $E$, and set $M = zA^{**}$, so that $(M,E)$ is a $(\Gamma,L^\infty(B))$-von Neumann algebra.

{\bf Claim.} The map $E$ is faithful and the $\Gamma$-action on $M$ is ergodic.

Denote by $p \in M$ the support projection of $E$ (so $p \in A^{**}$, $p \leq z$, and $z$ is the central support of $p$). Assume that $x \in M$ is such that $E(x^*x) = 0$. Since $E$ is equivariant and the action $\Gamma \actson A$ is a conjugacy action, we also have $E(u_g^*x^*x u_g) = \sigma_g^{-1}(E(x^*x)) = 0$, for every $g \in \Gamma$. This implies that $pu_g^*x^*xu_gp = 0$, and further $xu_gp = 0$ for every $g \in \Gamma$. Since the image of $A$ in $M$ is ultraweakly dense, we thus get $xyp = 0$ for every $y \in M$, and since the central support of $p$ in $M$ is $1$, this implies that $x = 0$. So $E$ is indeed faithful.

Let $q \in M^\Gamma$ be a $\Gamma$-invariant projection. Assume by contradiction that $q \notin \{0,1\}$. Then $q \in \cZ(M)$, and $E(q) \in L^\infty(B)^\Gamma = \C1$. Set $t \in [0, 1]$ so that $E(q) = t 1$. Since $E$ is faithful, we find that $t \in (0,1)$. We may thus define two ucp $\Gamma$-maps $E_1, E_2 : M \to L^\infty(B)$ by the formulae 
\[E_1(x) = \frac{1}{t} E(x q) \text{ and } E_2(x) = \frac{1}{1 - t} E(x (1-q)), \text{ for all } x \in M.\]
By construction, $E = t E_1 + (1 - t) E_2$. By extremality of $E_{|A}$, we find that $E_1$, $E_2$ and $E$ coincide on the image of $A$ in $M$. Now, since $E$ is normal on $M$ and $tE_1 \leq E$, $(1-t)E_2 \leq E$, we find that $E_1$ and $E_2$ are also normal on $M$. Thus these three maps coincide, which contradicts $E_1(q) = 1$, $E_2(q) = 0$. This finishes the proof of the claim.

Thanks to the claim, we may apply condition $(a)$. We find that either $\phi$ is invariant, or for almost every $b \in B$, for every $g \in \Gamma$, the states $g\theta(b) \in \cS(A)$ are pairwise singular. Let us prove that this later case is impossible.

In fact, we will check that a state $\psi$ on $A$ which is singular with respect to all its translates $g\psi$, $g \in \Gamma \setminus \{e\}$ is  the regular trace. In particular such a $\psi$ is $\Gamma$-invariant, so it cannot be singular.
Extend $\psi$ to a normal state on $A^{**}$, and denote by $q$ its support projection. By assumption, $\psi(q) = 1$ and $\psi(u_g q u_g^*) = (g^{-1}\psi)(q) = 0$ for every $g \in \Gamma \setminus \{e\}$. Therefore,
\begin{equation}\label{disjoint trick} 
\vert \psi(u_g) \vert = \vert \psi(u_gq) \vert = \vert \psi(u_gqu_g^*u_g) \vert \leq \psi(u_gqu_g^*)^{1/2}\psi(1)^{1/2} = 0.
\end{equation}
This shows that $\psi$ is the regular trace, as wanted.

{\em Classification of characters.}
Set $N = L^\infty(B)$. Take an extremal character $\tau$ on $\Gamma$, and denote by $M$ the corresponding GNS von Neumann algebra, which is a tracial factor. We consider the corresponding $(\Gamma,N)$-von Neumann algebra $(\cM,E) = ((N \ovt B(L^2(M)))^\Gamma, \id_N \ot \Phi)$ as defined in Example \ref{character structure}. By Lemma \ref{faithful character structure}, $E$ is faithful.

{\bf Claim.} The $\Gamma$-action on $\cM$ is ergodic.

This is where we use the condition that $(B,\nu)$ is metrically ergodic. By definition $\cM^\Gamma$ is the commutant of $1 \ot \pi_\tau(\Gamma)$ inside $\cM$ so it is equal to $(L^\infty(B) \ovt JMJ)^\Gamma$. This later algebra can be viewed as the algebra of $\Gamma$-equivariant measurable functions $B \to JMJ$, where the $\Gamma$-action on $JMJ$ is simply given by conjugacy by the unitaries $J\pi_\tau(g)J$, $g \in \Gamma$. Since $M$ is a tracial factor, $JMJ$ can be viewed as a subspace of its $L^2$-space, on which $\Gamma$ acts isometrically. So any equivariant function $B \to JMJ$ must be constant, equal to a scalar operator. Hence $\cM^\Gamma = \C1$ as desired.

So we are now in position to apply condition $(a)$. We find that either the structure map $E$ is invariant, or it is $\Gamma$-singular. We treat these two cases separately.

If $E$ is invariant, then $\cM = 1 \ovt M$. Indeed, assume that $E$ is invariant and take $f \in \cM$, which we view as a $\Gamma$-equivariant function $B \to B(L^2(M))$. Given $x,y \in M$, we have $1 \ot x, 1 \ot y \in \cM$, and hence $E((1 \ot y)^*f(1 \ot x))(b) = \Phi(y^*f(b)x) = \langle f(b)x\xi,y\xi\rangle$ does not depend on $b \in B$. Since $M$ is separable and $\xi$ is an $M$-cyclic vector, this implies that $f$ is essentially constant. Thus we find that $\cM = \cM \cap (1 \ovt B(L^2(M))) = 1 \ovt M$, as claimed. Since the action $\Gamma \curvearrowright B$ is amenable, $\cM$ is amenable and so is $M$.
In this case, $\tau$ is a von Neumann amenable character.

If $E$ is singular, we claim that $\tau$ is the regular character. Indeed, take a separable weakly dense C*-subalgebra $A_0 \subset \cM$ containing $1 \ot \pi(\Gamma)$. Denote by $\theta: B \to \cS(A_0)$ the measurable $\Gamma$-map corresponding to $E_{|A_0}$. Then computation \eqref{disjoint trick} tells us that for almost every $b \in B$, for every $g \in \Gamma$, $\theta(b)(1 \ot\pi(g)) = \delta_{g,e}$ and so $\theta(b) \circ (1 \ot \pi)$ is the regular character on $\Gamma$. In this case, the barycenter of these characters, which is exactly $\nu \circ E \circ (1 \ot \pi) = \tau$ is also the regular character, as claimed.
\end{proof}

We can also use condition (a) in Proposition \ref{top crit} to strengthen Proposition \ref{prop:URS}.

\begin{prop}
Let $\Gamma$ be a discrete group with trivial amenable radical. Let $(B,\nu)$ be an amenable ergodic $\Gamma$-space for which condition {\em (a)} in Proposition \ref{top crit} is satisfied.

Then any minimal action $\Gamma \curvearrowright X$ on a compact space is either topologically free or carries a $\Gamma$-invariant Borel probability measure.
\end{prop}
\begin{proof}
As in the proof of Proposition \ref{top crit}, we may choose an extremal measurable $\Gamma$-map $\theta : B \to \Prob(X)$. Set $\eta = \bary(\theta_\ast \nu) \in \Prob(X)$. Then $\eta$ is $\Gamma$-quasi-invariant and by minimality of $\Gamma \curvearrowright X$, the topological support of $\eta$ equals $X$. The $\Gamma$-ucp map $C(X) \to L^\infty(B)$ coming from $\theta$ extends to a well-defined faithful normal $\Gamma$-ucp map $F : L^\infty(X, \eta) \to L^\infty(B)$. By extremality of $\theta$, the nonsingular action $\Gamma \curvearrowright (X, \eta)$ is ergodic. Note that $\eta = \nu \circ F$.

By condition (a), $F$ is either $\Gamma$-invariant or $\Gamma$-singular. The former case implies that $\eta$ is a $\Gamma$-invariant Borel probability measure. Let us assume that $F$ is singular and argue that the action is topologically free. By definition, singularity of $F$ exactly means that $\theta(b) \perp g\theta(b)$, for every $g \in \Gamma \setminus \{e\}$, for almost every $b$. Fixing $g \in \Gamma \setminus \{e\}$, this condition further implies that $\theta(b)(\Fix(g)) = 0$, for almost every $b \in B$. Integrating this quantity w.r.t.\ $\nu$, we get $\eta(\Fix(g)) = 0$. Since $\eta$ has full support, this forces $\Fix(g)$ to have empty interior. So indeed the action is topologically free. 
\end{proof}

The above criterion can be adapted also for groups with a non-trivial amenable radical, but we need an extra stiffness assumption.

\begin{prop}\label{top crit 2}
Let $\Lambda$ be a countable group. Take a separable, metrically ergodic amenable $\Lambda$-space $(B,\nu)$ and write $N = L^\infty(B)$. The following conditions together imply that $\Lambda$ is charmenable.
\begin{description}
\item[(a')] For every separable, ergodic, faithful $(\Lambda,L^\infty(B))$-von Neumann algebra $(M,E)$, either $E$ is invariant or the maps $E_g,E_h$
given in Definition~\ref{def:Eg} are singular for every $g,h \in \Lambda$ such that $h^{-1}g\notin \Rad(\Lambda)$.
\item[(b)] Every measurable $\Lambda$-map $B \to \PD_1(\Rad(\Lambda))$ is essentially constant. 
\end{description}
\end{prop}
\begin{proof}
The proof follows the lines of the previous proposition. Let us explain the changes that come up.

In the fixed point property, condition $(a')$ ensures that for every $\Lambda$-map, $\theta: B \to \cS(C^*(\Lambda))$, the state $\phi := \bary(\theta_*\nu)$ is either invariant or for almost every $b \in B$, 
for every $g \in \Lambda \setminus \Rad(\Lambda)$, $\theta(b) \perp \theta(gb)$. 
In the later case, computation \eqref{disjoint trick} tells us that $\theta(b)(\pi(g)) = 0$ for almost every $b \in B$, for every $g \in \Lambda \setminus \Rad(\Lambda)$. 
Further, $\theta(b)$ is supported on $C^*(\Rad(\Lambda))$ for almost every $b \in B$. 
So in this case we may view $\theta$ as a $\Lambda$-map from $B$ into $\PD_1(\Rad(\Lambda))$. 
By condition (b) such a map must be constant, and hence its essential image must be a single $\Lambda$-invariant state. In particular, it cannot be singular with respect to its translates. So the second possibility is impossible, and $\phi$ is invariant.

The second part of the proof about classification of characters follows exactly the proof of Proposition \ref{top crit}.
\end{proof}

The following is a version of Proposition~\ref{top crit 2} which is somewhat easier to manage. 

\begin{prop} \label{prop:chmcrit3}
Let $\Lambda$ be a countable group and denote $\Gamma=\Lambda/\Rad(\Lambda)$. Take a separable, metrically ergodic amenable $\Gamma$-space $(B,\nu)$ and write $N = L^\infty(B)$. The following conditions together imply that $\Lambda$ is charmenable.
\begin{description}
\item[(a)] Every separable, ergodic, faithful $(\Gamma,L^\infty(B))$-von Neumann algebra $(M,E)$ is either invariant or $\Gamma$-singular.
\item[(b)] Every measurable $\Lambda$-map $B \to \PD_1(\Rad(\Lambda))$ is essentially constant. 
\end{description}
\end{prop}

\begin{proof}
Seeing $B$ as a $\Lambda$-space, it is still metrically ergodic and amenable, thus we only need to verify condition (a') of Proposition~\ref{top crit 2}.
We let $(M,E)$ be a separable, ergodic, faithful $(\Lambda,L^\infty(B))$-von Neumann algebra for which $E$ is not invariant and claim that $E$ is not invariant on the $(\Gamma,L^\infty(B))$-von Neumann algebra $(M^{\Rad(\Lambda)},E)$.
This will finish the proof, using Remark~\ref{rem:subsing} and condition (a). To prove this claim, it suffices to find a conditional expectation $E_0: M \to M^{\Rad(\Lambda)}$ which is $\Lambda$-equivariant, and such that $E = E \circ E_0$.

Fix a faithful normal state $\nu$ on $L^\infty(B)$, and consider the faithful normal and $\Rad(\Lambda)$-invariant state $\phi= \nu \circ E$ on $M$.
Take a generating probability measure $\mu$ on $\Rad(\Lambda)$ and note that $\phi$ is $\mu$-stationary. Consider the normal conditional expectation $E_\mu:M\to M^{\Rad(\Lambda)}$ given in Proposition~\ref{prop:stationary}(1). Then $E_\mu$ is the unique $\phi$-preserving conditional expectation $E_0$ onto $M^{\Rad(\Lambda)}$. In particular, it does not depend on the choice of $\mu$. 

Fix $g \in \Lambda$, and denote by $\alpha_g \in \Aut(\Rad(\Lambda))$ the automorphism obtained by restricting the conjugation action of $g$. Then denote by $\mu_g := (\alpha_g)_*\mu$ the push forward measure. Using the explicit construction of $E_\mu$, a direct computation shows that $E_0 = E_{\mu_g} = \sigma_g E_\mu \sigma_g^{-1} = \sigma_g E_0 \sigma_g^{-1}$. This proves that $E_0$ is $\Gamma$-equivariant.

By Proposition~\ref{prop:stationary}(2), for every normal state $\nu'$ on $L^\infty(B)$, $\nu' \circ E = \nu' \circ E\circ E_0$.
It follows that $E = E\circ E_0$, finishing the proof of the claim.  
\end{proof}

\begin{cor}\label{kill center}
Let $\Lambda$ be a countable group and assume that $\Rad(\Lambda)$ is either finite or central in $\Lambda$. 
Denote $\Gamma=\Lambda/\Rad(\Lambda)$ and let $(B,\nu)$ be a separable, amenable and metrically ergodic $\Gamma$-space and set $N = L^\infty(B)$. 
If condition (a) is satisfied then $\Lambda$ is charmenable.
\end{cor}

\begin{proof}
We need to verify condition (b) of Proposition~\ref{prop:chmcrit3}.
In case $\Rad(\Lambda)$ is central this follows at once from the ergodicity of $B$.
In case $\Rad(\Lambda)$ is finite, $\PD_1(\Rad(\Lambda))$ is finite dimensional and this follows from the metric ergodicity of $B$.
\end{proof}

The rest of the paper is devoted to prove that these conditions (a) and (b) are satisfied in the cases of interest.

\section{$(G, N)$-structures, lattices with dense projections and induction}
\label{sect:vN}

In this section, we are interested in the following problem. Assume that $\sigma : \Gamma \curvearrowright X$ is an action of a discrete countable group on a topological vector space $X$ with some extra structure (typically $X$ is a Hilbert space or a von Neumann algebra). Let $\iota: \Gamma \to G_1$ be a group homomorphism into an lcsc group $G_1$ with dense image. Then  we want to give an algebraic description of the set of elements $x \in X$ such that the orbit map $ \Gamma \to X : \gamma \mapsto \sigma_\gamma(x)$ factors to a map defined on $\iota(\Gamma)$, which extends continuously to a map $G_1 \to X$. 

In our setting, $\Gamma$ will be a lattice in an lcsc group $G$ and the morphism $\iota$ extends to a continuous homomorphism $G \to G_1$. In this case, we shall identify this continuity space with a fixed point set in the induced action.

\subsection{Continuity vectors for unitary representations} \label{ss:contvectors}

Let $\Gamma < G$ be a lattice in an lcsc group $G$, let $G_1$ be a quotient of $G$ with kernel $G_2$. Denote by $\iota: G \to G_1$ the quotient map and assume that $\iota(\Gamma)$ is dense in $G_1$.

Let $\pi: \Gamma \to \cU(H)$ be any unitary representation, and denote by $(\tpi,\tH)$ the induced unitary representation of $G$.

We say that a vector $v \in H$ is $\iota$-{\em continuous} if $\lim_n \|\pi(\gamma_n)v - v\| = 0$ for any sequence $(\gamma_n)_{n \in \N}$ in $\Gamma$ such that $\iota(\gamma_n) \to e$ in $G_1$. We denote by $H_\iota$ the set of $\iota$-continuous vectors. Because the action of $\Gamma$ on $H$ is isometric, one checks that $H_\iota$ is a closed $\Gamma$-invariant subspace of $H$. Moreover, for any $v \in H_\iota$, there exists a unique continuous map $c_v : G_1 \to H$ such that $\pi(\gamma)v = c_v(\iota(\gamma))$ for every $\gamma \in \Gamma$. In other words, we may extend $\pi : \Gamma \to \cU(H_\iota)$ to a continuous unitary representation $\pi : G \to \cU(H_\iota)$ that factors through $G_1$ and that satisfies $\pi(g)v = c_v(\iota(g))$ for every $g \in G_1$ and every $v \in H_\iota$.

\begin{prop}\label{continuity rep}
We keep the notation as above. There is a $G$-equivariant surjective isometry
\[\kappa : H_\iota \to (\tH)^{G_2}.\] 
\end{prop}
\begin{proof}
Let us view $\tH$ as the Hilbert space of measurable maps $f : G \to H$ such that
\begin{itemize}
\item [$(\rm i)$] For every $\gamma \in \Gamma$ and almost every $g \in G$, $f(g \gamma) = \pi(\gamma^{-1})f(g)$.
\item [$(\rm ii)$] $\|f\|^2 = \int_{G/\Gamma} \|f(g)\|^2 \, {\rm d}m_{G/\Gamma}(g\Gamma) < +\infty$.
\end{itemize}
In this description, let us check that the map $\kappa: H_\iota \to \tH$ defined by $\kappa(v)(g) := \pi(g^{-1})v$, for every $v \in H_\iota$, $g \in G$, suits us. It is indeed isometric and $G$-equivariant, and it indeed ranges into $\tH^{G_2}$ by definition of $H_\iota$. It remains to prove that $\kappa$ is surjective. Fix $f \in \tH^{G_2}$. 

{\bf Claim.} Every essential value of $f$ is an element of $H_\iota$.

Indeed, let $v$ be any essential value of $f$ and take a sequence $(\gamma_n)_{n \in \N}$  in $\Gamma$ such that $\iota(\gamma_n) \to e$ in $G_1$. We want to check that $\lim_n \|\pi(\gamma_n)v - v\| = 0$. We may find elements $h_n \in G_2$ such that $\gamma_nh_n \to e$ in $G$. Take $\eps > 0$.
By assumption, the set $A_\eps = \{ x \in G \mid \Vert f(x) - v \Vert < \eps\}$ has positive measure in $G$. Since $\gamma_nh_n \to e$ in $G$, we may find $n \in \N$ large enough so that $A_\eps \cap (A_\eps \cdot (\gamma_nh_n)^{-1})$ has positive measure. As an element of $\tH^{G_2}$, the function $f: G \to H$ is left $G_2$-invariant (so right $G_2$-invariant as well since $G_2$ is normal in $G$) and right $\Gamma$-equivariant. Thus for every $g \in G$ and every $n \in \N$, we have $f(g(\gamma_nh_n)) = f(g\gamma_n) = \pi(\gamma_n^{-1})f(g)$. So for $n \in \N$ large enough, choosing $g \in A_\eps \cap (A_\eps \cdot (\gamma_nh_n)^{-1})$, we have
\[\Vert v - \pi(\gamma_n)v\Vert  \leq \Vert v - f(g)\Vert + \Vert f(g) - \pi(\gamma_n)v\Vert  \leq \Vert v - f(g)\Vert  + \Vert f(g\gamma_nh_n) - v \Vert \leq 2\eps.\]
As $\eps > 0$ can be arbitrarily small, this finishes the proof of the claim. 

Using this claim, we may modify $f$ on a null set if necessary to view it as an $H_\iota$-valued map. Then the measurable function $G \to H_\iota : g \mapsto \pi(g)(f(g))$ is well-defined, it is $G_2$-invariant and also right $\Gamma$-invariant. Since the product set $G_2 \Gamma$ is dense in $G$, this implies that the above measurable function is essentially constant. If we denote by $v \in H_\iota$ its essential value, we find that $f = \kappa(v)$. 
\end{proof}

\begin{rem}
In fact a similar result holds for more general metric $\Gamma$-spaces and $L^p$-induction, for arbitrary $p \in [1,\infty)$. 
We will not elaborate on this further here, as we will only make use of the above setting.
\end{rem}

\subsection{Continuity points in $(\Gamma,N)$-algebras}\label{ss:continuity-vN}

We now investigate the case of von Neumann algebras. We start with the following general terminology.

\begin{defn}\label{defn:continuous}
Consider a countable discrete group $\Gamma$, an lcsc group $G_1$ and a group homomorphism $\iota: \Gamma \to G_1$ with dense range. Let $M$ be a $\Gamma$-von Neumann algebra. We say that an element $x \in M$ is {\em $\iota$-continuous} if $\sigma_{\gamma_n}(x) \to x$ $*$-strongly in $M$ for any sequence $(\gamma_n)_{n \in \N}$ in $\Gamma$ such that $\iota(\gamma_n) \to e$ in $G_1$.
\end{defn}

When the map $\iota$ is obviously understood from $G_1$, we will also use the terminology $G_1$-continuous, instead of $\iota$-continuous.

From now on, we denote by $G = G_1 \times G_2$ a product of two lcsc groups and $\Gamma < G$ a lattice with dense projections.
For every $i \in \{1, 2\}$, we denote by $p_i : G \to G_i$ the projection map and for consistency of notation with the previous paragraphs, we denote by $\iota$ the restriction of $p_1$ to $\Gamma$. 

If a $\Gamma$-von Neumann algebra $M$ carries a $\Gamma$-invariant faithful normal state, then we can use metric considerations as in the previous subsection to identify the set of $\iota$-continuous elements with a fixed point subalgebra in the induced von Neumann algebra. This was observed in \cite[Section 4]{CP13} (see also the comment after \cite[Proposition 3.1]{Pe14}). Unfortunately in the cases of interest to us, no such state is assumed to exist. Instead we have a specific stationary state, coming from a Furstenberg-Poisson boundary of $G$. We aim to provide the analogous conclusion in this weaker setting.

For $i = 1,2$, choose an admissible Borel probability measure $\mu_i \in \Prob(G_i)$ and denote by $(B_i,\nu_i)$ the Furstenberg-Poisson boundary of $(G_i, \mu_i)$. Then the product $G$-space $(B,\nu) := (B_1,\nu_1) \times (B_2,\nu_2)$ is the Furstenberg-Poisson boundary of $G$ with respect to the product measure $\mu := \mu_1 \otimes \mu_2 \in \Prob(G)$ (see \cite[Corollary 3.2]{BS04}). We will write $N_1 = L^\infty(B_1)$, $N_2 = L^\infty(B_2)$ and $N = L^\infty(B) = N_1 \ovt N_2$. 

Observe that if $(M,E)$ is a $(\Gamma,N)$-von Neumann algebra then $E$ maps $\iota$-continuous elements in $M$ to $\iota$-continuous elements in $N$. We can therefore take advantage of the fact that $N$ is already a $G$-algebra. The following lemma will play an essential role.

\begin{lem}\label{exmp continuity points}
The set of $\iota$-continuous elements in $N$ is equal to $N_1 \ot 1$.
\end{lem}
\begin{proof}
Let $f \in N$ be a $\iota$-continuous element in $N$. We view $f$ as an $N_1$-valued function on $B_2$, $f \in L^\infty(B_2, N_1)$ and we choose an essential value $y \in N_1$ of $f$. 

For $\varepsilon > 0$, the set $E_\varepsilon = \{b \in B_2 \mid \|f(b) - y\|_{\nu_1} < \varepsilon\}$ has positive measure in $B_2$. By \cite[Lemma 5.1]{Pe14}, there exists a sequence $(\gamma_n)_{n \in \N}$ in $\Gamma$, so that $\iota(\gamma_n) \to e$ in $G_1$ and $\nu_2(p_2(\gamma_n) E_\varepsilon) \to 1$. Because $f$ is $\iota$-continuous, we find
\begin{align*}
\|f - y \ot 1\|_\nu^2  &= \lim_n \|\sigma_{\gamma_n}(f) - y \ot 1\|_\nu^2 \\
&= \lim_n \int_{B_2} \| \sigma_{p_1(\gamma_n)}(f(p_2(\gamma_n)^{-1} b))- y\|_{\nu_1}^2 \dd\nu_2(b) \\
&= \lim_n \int_{B_2} \| f(p_2(\gamma_n)^{-1} b)- \sigma_{\iota(\gamma_n)}^{-1}(y)\|_{\nu_1 \circ \sigma_{\iota(\gamma_n)}}^2 \dd\nu_2(b) \\
&= \lim_n \int_{B_2} \| f(p_2(\gamma_n)^{-1} b) - y\|_{\nu_1}^2 \dd\nu_2(b).
\end{align*}
This latter integral can split into two parts: the integral over $p_2(\gamma_n) E_\eps$, where the integrand is less than $\eps^2$, and the integral over the complementary set, whose measure goes to $0$ as $n$ goes to infinity (and where the integrand is bounded by $(2\Vert f \Vert)^2$). So we find that $\|f - y \ot 1\|_\nu^2 \leq \eps^2$. Since $\eps > 0$ can be arbitrary, we reach the desired conclusion that $f = y \otimes 1 \in N_1 \otimes 1$.
\end{proof}

\begin{thm}\label{thm:continuity-vN}
Let $(M,E)$ be a faithful $(\Gamma,N)$-von Neumann algebra. Denote by $M_1 \subset M$ the subset of $G_1$-continuous elements with respect to $\iota: \Gamma \to G_1$. 

Then $M_1 \subset M$ is a globally $\Gamma$-invariant von Neumann subalgebra and the action $\Gamma \curvearrowright M_1$ extends to a continuous action $G \curvearrowright M_1$ such that $G_2$ acts trivially.
\end{thm}
\begin{proof}
One easily checks that $M_1$ is $\Gamma$-invariant. Moreover, $M_1$ is a $\ast$-subalgebra of $M$ simply because the multiplication map $M \times M \to M : (x, y) \mapsto xy$ is $\ast$-strongly continuous on uniformly bounded sets. 
The following claims prove the remaining statements.

{\bf Claim 1.} For any $x \in M_1$, the orbit map $\Gamma \to M : \gamma\mapsto \sigma_\gamma(x)$ extends to a continuous map $G \to M$, which only depends on the first variable and takes values in $M_1$.

The extension map is constructed by the classical extension argument for uniformly continuous maps into complete spaces, but we do it by hand. Take $x \in M_1$. 
Let $g \in G$ and take a sequence $(\gamma_n)_{n \in \N}$ in $\Gamma$ such that $\iota(\gamma_n) \to p_1(g)$ in $G_1$. We prove that $(\sigma_{\gamma_n}(x))_{n \in \N}$ converges $\ast$-strongly in $M$. For this, consider the faithful normal state $\phi = \nu \circ E \in M_\ast$ and recall that the strong topology on bounded sets of $M$ is given by the norm $\Vert \cdot \Vert_\phi$. For all $m, n \in \N$, we have 
\begin{align*}
\Vert \sigma_{\gamma_n}(x) - \sigma_{\gamma_m}(x) \Vert_\phi^2 & = \nu \circ  E ((\sigma_{\gamma_n}(x) - \sigma_{\gamma_m}(x))^*(\sigma_{\gamma_n}(x) - \sigma_{\gamma_m}(x)))\\
& = \nu \circ \sigma_{\gamma_m} \circ E (y_{n,m}),
\end{align*}
where $y_{n,m} = (\sigma_{\gamma_m^{-1}\gamma_n}(x) - x)^*(\sigma_{\gamma_m^{-1}\gamma_n}(x) - x)$. Since $x \in M_1$ and $E$ is normal, $E(y_{n,m})$ converges ultraweakly to $0$ as $n,m \to \infty$. Moreover, since $M_1$ is a $\Gamma$-invariant *-subalgebra of $M$, $y_{n,m} \in M_1$ and thus Lemma \ref{exmp continuity points} implies that $E(y_{n,m}) \in N_1 \ot 1$, for all $n, m \in \N$. In particular, we find 
\[\sigma_{\gamma_m} \circ E(y_{n,m}) = \sigma_{\iota(\gamma_m)} \circ E (y_{n,m}).\]
Since $\iota(\gamma_m) \to p_1(g)$ in $G_1$ and since the action map $ G \times N \to N$ is ultraweakly continuous, we conclude that $\sigma_{\gamma_m} \circ E(y_{n,m}) \to \sigma_{p_1(g)}(0) = 0$, ultraweakly in $N$ as $m \to \infty$ and $n \to \infty$.

This shows that the uniformly bounded sequence $(\sigma_{\gamma_n}(x))_{n \in \N}$ is $\|\cdot\|_\phi$-Cauchy and hence strongly converges to some $y \in M$. Applying the same argument with $x^*$ instead of $x$, we see that the sequence $(\sigma_{\gamma_n}(x))_{n \in \N}$ is $\ast$-strongly convergent to $y \in M$. The above computation also applies to show that the $\ast$-strong limit $y \in M$ does not depend on the choice of the sequence $(\gamma_n)_{n \in \N}$ but only on $p_1(g) \in G_1$. Therefore,  we may define $\sigma_g(x) = y$, which thus only depends on the first variable $p_1(g) \in G_1$. The independence on the sequence $(\gamma_n)_{n \in \N}$ also implies that the orbit map $g \in G \mapsto \sigma_g(x)$ is strongly continuous.

Let us check that for every $g \in G$, $\sigma_g(x) \in M_1$. Indeed, let $g \in G$ and $(\gamma_n)_{n \in \N}$ any sequence in $\Gamma$ such that $\iota(\gamma_n) \to e$ in $G_1$. We have to show that $\sigma_{\gamma_n}(\sigma_g(x)) \to \sigma_g(x)$ $\ast$-strongly. For any $\eps > 0$, we may find a neighborhood $U \subset G_1$ of $p_1(g)$ such that $\Vert \sigma_\gamma(x) - \sigma_g(x) \Vert_\phi < \eps$ for all $\gamma \in \Gamma$ such that $\iota(\gamma) \in U$. Take a neighborhood $U_1 \subset G_1$ of $e$ and a neighborhood $U_2 \subset G_1$ of $p_1(g)$ such that $U_1 U_2 \subset U$. Fix $n \in \N$ large enough so that $\iota(\gamma_n) \in U_1$. By definition of $\sigma_g(x)$, we may find $\gamma \in \Gamma$ such that $\iota(\gamma) \in U_2$ and
\[\Vert \sigma_{\gamma_n} (\sigma_\gamma(x) - \sigma_g(x))\Vert_\phi =  \|\sigma_\gamma(x) - \sigma_g(x)\|_{\phi \circ \sigma_{\gamma_n}} < \eps.\]
Then we have
\begin{align*}
\Vert \sigma_{\gamma_n}(\sigma_g(x)) - \sigma_g(x)\Vert_\phi & \leq \Vert \sigma_{\gamma_n} (\sigma_g(x) - \sigma_\gamma(x))\Vert_\phi + \Vert \sigma_{\gamma_n}(\sigma_{\gamma}(x)) - \sigma_g(x)\Vert_\phi\\
& \leq \eps + \Vert\sigma_{\gamma_n\gamma}(x) - \sigma_g(x)\Vert_\phi.
\end{align*}
But since $\iota(\gamma_n \gamma) \in U$, the last term above is also bounded by $\eps$, and hence for all $n \in \N$ large enough, we get
\[\Vert \sigma_{\gamma_n}(\sigma_g(x)) - \sigma_g(x)\Vert < 2\eps.\]
This proves that $\sigma_{\gamma_n}(\sigma_g(x)) \to \sigma_g(x)$ strongly. Applying the same reasoning to $x^* \in M_1$, we obtain $\sigma_{\gamma_n}(\sigma_g(x)) \to \sigma_g(x)$ $\ast$-strongly. This proves that $\sigma_g(x) \in M_1$ and finishes the proof of Claim 1.

{\bf Claim 2.} $M_1 \subset M$ is a von Neumann subalgebra and $\sigma : G \curvearrowright M_1$ is a continuous action.

Indeed, let $x \in (M_1)\dpr$ and take a sequence $(\gamma_n)_n$ in $\Gamma$ such that $\iota(\gamma_n) \to e$ in $G_1$. Fix $\eps > 0$ and take $x_0 \in M_1$ such that $\Vert x - x_0\Vert_\phi < \eps$. Since $(x - x_0)^*(x - x_0)$ is in the weak closure of $M_1$, Lemma \ref{exmp continuity points} implies that $E((x - x_0)^*(x - x_0)) \in N_1 \ot 1$, i.e.\ this element is $\iota$-continuous in $N$. In particular, $\lim_n \sigma_{\gamma_n}(E((x - x_0)^*(x - x_0))) = E((x - x_0)^*(x - x_0))$. Applying $\nu$, we find 
\[\limsup_n \Vert \sigma_{\gamma_n}(x - x_0)\Vert_\phi^2 = \limsup_n \nu \circ \sigma_{\gamma_n} \circ E((x - x_0)^*(x - x_0)) = \Vert x - x_0\Vert_\phi^2 < \eps^2. \]
This allows to compute 
\[\limsup_n \Vert \sigma_{\gamma_n}(x) - x \Vert_\phi \leq \limsup_n \left(\Vert \sigma_{\gamma_n}(x - x_0) \Vert_\phi + \Vert \sigma_{\gamma_n}(x_0) - x_0 \Vert_\phi + \Vert x_0 - x \Vert_\phi \right) < 2\eps.\]
As $\eps > 0$ can be arbitrarily small, this shows that $\sigma_{\gamma_n}(x) \to x$ strongly. Applying the same reasoning to $x^* \in M_1$, we obtain that $\sigma_{\gamma_n}(x) \to x$ strongly. So $x \in M_1$ and thus $M_1$ is indeed a von Neumann algebra. The fact the action $\sigma : G \curvearrowright M_1$ is continuous follows from Claim 1 and \cite[Proposition X.1.2]{Ta03a}. 
\end{proof}

\begin{thm}\label{continuous elts}
Keep the notation $\Gamma < G = G_1 \times G_2$, $\iota$, $N = N_1 \ovt N_2$ as above. Let $(M,E)$ be a separable faithful $(\Gamma,N)$-von Neumann algebra. Denote by $(\tM,\tE)$ the induced $(G,\tN)$-algebra as defined in Example \ref{induced structure}. 

The algebra $M_1 \subset M$ of $\iota$-continuous elements identifies with the fixed point algebra $\tM^{G_2}$. More precisely, there is a $G$-equivariant surjective isomorphism $\kappa: M_1 \to \tM^{G_2}$ such that $(E_N \circ \tE) \circ \kappa = E$, where $E_N : \tN \to N$ is as defined in Example \ref{induced structure}.
\end{thm}
\begin{proof}
We view $\tM = (L^\infty(G) \ovt M)^{(\rho \ot \sigma)(\Gamma)}$ as the algebra of $\Gamma$-equivariant functions from $G$ to $M$ (with respect to the right $\Gamma$-action on $G$).
We then define the map $\kappa: M_1 \to \tM$ by the formula 
\[\kappa(x)(g) = \sigma_g^{-1}(x) \in M, \text{ for all } x \in M_1, g \in G.\] 
This map is clearly $G$-equivariant, so it must range into $\tM^{G_2}$. It is also obvious that $\kappa$ is injective; let us prove that it is surjective.

Let $f \in \tM^{G_2}$. Proceeding as in the proof of Proposition \ref{continuity rep}, in order to show that $f$ is in the range of $\kappa$, it suffices to show that any essential value $y$ of $f$ is an element of $M_1$. 

Lemma \ref{lem:fixed-point} below implies that $\tN^{G_2} \subset L^\infty(G) \ovt N_1 \ovt 1$, so $\tE(f) \in L^\infty(G) \ovt N_1 \ovt 1$. Since $\tE$ is equal to $\id \ot E$, we deduce that $E(f(g)) \in N_1 \ovt 1$, for almost every $g \in G$. In particular, $E(y) \in N_1 \ovt 1$. We may apply the same reasoning to $f^*f \in (\tM)^{G_2}$ and deduce that $E(y^*y) \in N_1 \ovt 1$. This fact will be useful, but we need more.

{\bf Claim.} For almost every $g,h \in G$, we have $E(f(g)^*f(h)) \in N_1 \ovt 1$.

Indeed, the measurable function of two variables $F: G \times G \to N : (g,h)  \mapsto E(f(g)^*f(h))$ is $G_2 \times G_2$-invariant and it is $\Gamma$-equivariant in the sense that $F(g\gamma,h\gamma) = \sigma_\gamma^{-1}(F(g,h))$, for all almost all $g,h \in G$ and all $\gamma \in \Gamma$. The claim now follows from Lemma \ref{lem:fixed-point} below.

Using this claim and the observations preceding it, we find that for almost every $g \in G$, $E((f(g) - y)^*(f(g) - y)) \in N_1 \ovt 1$ and so $E((f(g) - y)^*(f(g) - y))$ is $G_2$-invariant. Let $(\gamma_n)_{n \in \N}$  be any sequence in $\Gamma$ such that $\iota(\gamma_n) \to e$ in $G_1$. We now show that $\sigma_{\gamma_n}(y) \to y$ $\ast$-strongly in $M$. Let $\eps > 0$ and consider the set of positive measure 
\[A= \{ g \in G \mid \Vert f(g) - y \Vert_\phi < \eps \}.\]
Choose $n \in \N$ large enough so that the intersection $A \cap (A \cdot \iota(\gamma_n))$ has positive measure and pick an element $g \in A \cap (A \cdot \iota(\gamma_n))$ such that $E((f(g) - y)^*(f(g) - y))$ is $G_2$-invariant. We may also assume that $n$ is large enough so that $\Vert \nu - \nu \circ \sigma_{\iota(\gamma_n)}  \Vert \cdot (2\Vert f \Vert)^2 < \eps^2$.

Then on the one hand, we have $g\iota(\gamma_n)^{-1} \in A$, and $\Vert f(g \gamma_n^{-1}) - y \Vert_\phi = \Vert f(g\iota(\gamma_n)^{-1}) - y \Vert_\phi < \eps$. On the other hand, we have
\[ \Vert f(g \gamma_n^{-1}) - \sigma_{\gamma_n}(y) \Vert_\phi^2 = \Vert \sigma_{\gamma_n}(f(g) - y)\Vert^2_\phi = \nu \circ \sigma_{\gamma_n} \circ E((f(g) - y)^*(f(g) - y)).\]
By our choice of $g$, $E((f(g) - y)^*(f(g) - y))$ is $G_2$-invariant and hence we may continue our computation
\begin{align*}
\Vert f(g \gamma_n^{-1}) - \sigma_{\gamma_n}(y) \Vert_\phi^2 & =  \nu \circ \sigma_{\iota(\gamma_n)} \circ E((f(g) - y)^*(f(g) - y))\\
& \leq \Vert f(g) - y\Vert_\phi^2 + \Vert \nu - \nu \circ \sigma_{\iota(\gamma_n)}\Vert \cdot (2\Vert f \Vert)^2\\
& < 2\eps^2.
\end{align*}
In conclusion, we see that
\[\Vert y - \sigma_{\gamma_n}(y) \Vert_ \phi \leq \Vert y - f(g\gamma_n^{-1}) \Vert_ \phi + \Vert f(g \gamma_n^{-1}) - \sigma_{\gamma_n}(y) \Vert_\phi < (1 + \sqrt{2})\eps.\]
This proves that $\sigma_{\gamma_n}(y) \to y$ strongly in $M$. Applying the same reasoning to $y^* \in M$ which is an essential value of $f^* \in (\widetilde M)^{G_2}$, we obtain $\sigma_{\gamma_n}(y) \to y$ $\ast$-strongly in $M$. So $y \in M_1$, as desired.

Finally, the equality $E_N \circ \tE \circ \kappa = E$ can be verified by making the map $E_N$ explicit.
\end{proof}

We used the following technical result.

\begin{lem}\label{lem:fixed-point}
Let  $\cN = L^\infty(G) \ovt L^\infty(G) \ovt N$ and define the action $\sigma : G_2 \times G_2 \times \Gamma \curvearrowright \cN$ by $\sigma_{(g,h,\gamma)} = \lambda_g\rho_{\gamma} \ot \lambda_h\rho_{\gamma}  \ot \sigma_\gamma$ for $g,h \in G_2$, $\gamma \in \Gamma$.
Then we have
\[\cN^{G_2 \times G_2 \times \Gamma} \subset L^\infty(G) \ovt L^\infty(G) \ovt N_1 \ovt 1_{B_2}.\]
In particular $\tN^{G_2} \subset L^\infty(G) \ovt N_1 \ovt 1_{B_2}$.
\end{lem}
\begin{proof}
Set $\mathcal P = L^\infty(G/\Gamma) \ovt L^\infty(G) \ovt N$ and define the action $\beta : G_2 \times G_2\curvearrowright \mathcal P$ by $\beta_{(g, h)} = \lambda_g \otimes \lambda_h \rho_g \otimes \sigma_g$ for $g, h \in G_2$. Define the unital $\ast$-isomorphism $\Xi : \cN^\Gamma \to \mathcal P$ by the formula
\[\Xi(F) (g\Gamma, h) = \sigma_g(F(g, hg)), \text{ for every } F \in \cN^\Gamma, \text{ almost every } (g, h) \in G \times G.\]
 One checks that the isomorphism $\Xi$ is onto and intertwines the action $G_2 \times G_2 \curvearrowright \mathcal N^\Gamma$ with the action $\beta : G_2 \times G_2 \curvearrowright \mathcal P$. 
 
Let now $F \in \cN^{G_2 \times G_2 \times \Gamma}$. Then $\Xi(F) \in  L^\infty(G/\Gamma) \ovt L^\infty(G_1) \ovt N$ and $\Xi(F)$ invariant under the automorphisms $\lambda_g \ot \id_{G_1} \ot \sigma_g$ for all $g \in G_2$. Since $\Gamma < G_1 \times G_2$ is a lattice with dense projections, the pmp action $G_2 \curvearrowright G/\Gamma$ is ergodic and \cite[Corollary 2.18]{BS04} implies that the diagonal action  $G_2 \actson G/\Gamma \times B_2$ is ergodic. This further implies that $\Xi(F) \in L^\infty(G/\Gamma) \ovt L^\infty(G) \ovt N_1 \ovt 1_{B_2}$ which in turn implies that $ F \in L^\infty(G) \ovt L^\infty(G) \ovt N_1 \ovt 1_{B_2}$.
\end{proof}

Combining the above result with Lemma \ref{lem:structure-product} we obtain the following key theorem.

\begin{thm}\label{non-trivial G points}
Take $k \geq 2$ and a product $G = G_1 \times \cdots \times G_k$ of $k$ lcsc groups. For every $1 \leq i \leq k$, choose an admissible Borel probability measure $\mu_i \in \Prob(G_i)$ and denote by $(B_i,\nu_i)$ the Furstenberg-Poisson boundary of $(G_i, \mu_i)$. Then the product $G$-space $(B,\nu) := (B_1,\nu_1) \times \cdots \times  (B_k,\nu_k)$ is the Furstenberg-Poisson boundary of $G$ with respect to the product measure $\mu := \mu_1 \ot \cdots \ot \mu_k \in \Prob(G)$ (see \cite[Corollary 3.2]{BS04}). Set $N = L^\infty(B,\nu)$.

Take a lattice with dense projections $\Gamma < G$ and a faithful $(\Gamma,N)$-von Neumann algebra $(M,E)$. If $E$ is not $\Gamma$-invariant, then there exists $1 \leq i \leq k$, such that the von Neumann subalgebra $M_i$ of $G_i$-continuous elements in $M$ is nontrivial,
$E_{|M_i}$ is not $G_i$-invariant and its image is in $L^\infty(B_i)\le L^\infty(B)$. 
\end{thm}
\begin{proof}
Following Example \ref{induced structure}, denote by $(\tM,\tE)$ the induced $(G,\tN)$-structure and view $E_N \circ \tE$ as a $(G,N)$-structure. If $E$ is not $\Gamma$-invariant, Lemma \ref{invariant induced} implies that $E_N \circ \tE$ is not $G$-invariant. Since $(B,\nu)$ is the Furstenberg-Poisson boundary of $G$, Example \ref{stationary structure} further implies that the faithful $\mu$-stationary state $\phi = \nu \circ E_N \circ \tE$ is not $G$-invariant on $\tM$.

In particular, there exists $i$ such that $\phi$ is not $G_i$-invariant. Gather the factors of $G$ to write it as a product of two groups $G_i \times H_i$. By Lemma \ref{lem:structure-product}, we find that $\phi$ is not $G_i$-invariant on $\tM^{H_i}$. Thanks to the observations in Example \ref{stationary structure}, this amounts to saying that $E_N \circ \tE$ is not $G_i$-invariant on $\tM^{H_i}$. By Theorem \ref{continuous elts}, this exactly means that $E$ is not invariant on the algebra of $G_i$-continuous elements $M_i$.
We saw in Lemma \ref{exmp continuity points} and the comment preceding it that indeed $E$ maps $M_i$ into $L^\infty(B_i)$.
\end{proof}

\section{Proofs of charmenability} \label{s:proof}

In this section, we prove Theorem~\ref{thm:SD} and Theorem~\ref{thm:Tree}, as well as Proposition~\ref{ss:proof}
which consists of the first half of
Theorem~\ref{thm:AG}.

%
%

\subsection{Arithmetic groups} \label{ss:proof}

The main result of this subsection is the following proposition.

\begin{prop} \label{SAF}
Let $K$ be a global field and ${\bf G}$ a connected non-commutative $K$-almost simple $K$-algebraic group.
If $\Gamma\leq {\bf G}(K)$ is an $S$-arithmetic subgroup of a product type
then $\Gamma$ is charmenable.
\end{prop}


For the proof we need to establish the following freeness result.

\begin{lem}\label{freeness1}
Let $k$ be a local field and $\G$ a connected $k$-almost simple $k$-algebraic group. 
Let ${\bf H} \lneq \G$ be a proper $k$-subgroup and let $G = \G(k)$, $H = {\bf H}(k)$. We endow $G/H$ with the unique $G$-invariant class of Radon measures.  Then for every $g \in G \setminus \cZ(G)$, for almost every $w \in G/H$, we have $gw \neq w$.
\end{lem}

The proof of Lemma~\ref{freeness1} in turn relies on the following preliminary result.

\begin{lem} \label{lem:nullsetsG/H}
Let $k$ be a local field and $\bar{k}$ an algebraically closed field extension of $k$.
Let $\G$ be a connected $k$-algebraic group and denote $G = \G(k)$.
Let $\bH \leq \G$ be a $k$-algebraic subgroup
and denote $H = \bH(k)$.
We endow $G/H$ with the unique $G$-invariant class of Radon measures. 
We let ${\bf U}$ be a closed proper subvariety of $\G/\bH$, ${\bf U}\subsetneq \G/\bH$.
Then, considering $G/H$ as a subset of $\G/\bH(k)\subset \G/\bH(\bar{k})$, we have that ${\bf U}(\bar{k})\cap G/H$ is a null set in $G/H$.
\end{lem}

\begin{proof}
Denote by $\pi:\G(\bar{k}) \to \G/\bH(\bar{k})$ the natural map and by $\pi_k:G \to G/H \subset \G/\bH(\bar{k})$ its restriction to the $k$-points. 
It is a general fact about lcsc groups that a subset of $G/H$ is null if and only if its preimage in $G$ is null. Let us check that indeed $\pi_k^{-1}({\bf U}(\bar{k}))$ is null in $G$.

We denote by ${\bf V}$ the preimage of ${\bf U}$ in $\G$ and observe that 
this is a closed proper subvariety of $\G$ satisfying ${\bf V}(k) =\pi_k^{-1}({\bf U}(\bar{k}))$. 
By \cite[Theorem AG.14.4]{Bo91}, the Zariski closure ${\bf V}_0$ of ${\bf V}(k)$ in $\G$ is a $k$-subvariety of $\G$, contained in ${\bf V}$. So in particular ${\bf V}_0$ is a proper $k$-subvariety of $\G$, which satisfies ${\bf V}_0(k)={\bf V}(k)$.
Since $\G$ is connected, \cite[Proposition I.2.5.3(ii)]{Ma91} implies that ${\bf V}_0(k)$ is indeed a null set in $G$.
%
\end{proof}

\begin{proof}[Proof of Lemma~\ref{freeness1}]
Fix $g \in G\setminus \cZ(G)$.
Note that $g$ acts non-trivially on $G/H$, otherwise $g$ would belong to the normal subgroup $\bigcap_{x \in G} xHx^{-1}$, the Zariski closure of which is a proper normal $k$-subgroup ${\bf N}$ of $\G$. 
Since $\G$ is $k$-almost simple, we have ${\bf N} \subset \cZ(\G)$, forcing $g \in \cZ(G)$, which we excluded.
Hence the subvariety   ${\bf U}$ of fixed points of $g$ in $\G/\bH$ is proper, and we conclude by applying Lemma~\ref{lem:nullsetsG/H}.
\end{proof}

We now have set up all the tools we need to prove Proposition~\ref{SAF}.

\begin{proof}[Proof of Proposition~\ref{SAF}]
By Proposition~\ref{prop:stable}, we assume as we may that the set $S$ is finite. 
We consider the finite set $I$ of places $v$ of $K$ such that the image of $\Gamma$ in ${\bf G}(K_v)$ is unbounded. 

For each $i\in I$, we denote by
\begin{itemize}
\item $k_i$ the completion of $K$ with respect to the place $i$;
\item ${\G}_i$ the algebraic group ${\bf G}$ viewed as a $k_i$-group;
\item ${\bf P}_i$ a minimal $k_i$-parabolic subgroup of ${\bf G}_i$;
\item $G_i < \G_i(k_i)$ the closure of the image of $\Gamma$ in $\G_i(k_i)$, and $P_i := {\bf P}_i(k_i) \cap G_i$.
\end{itemize}
Note that $\G_i(k_i)^+ \leq G_i \leq \G_i(k_i)$, by the strong approximation theorem (see \cite[Theorem II.6.8]{Ma91}). 
Therefore we may apply Example~\ref{Mautner algps}, and find that $G_i$ acts transitively on ${\bf G}_i(k_i)/{\bf P}_i(k_i)$
with stabilizer $P_i$ and the pair $(G_i,P_i)$ is stably self-normalizing and it has the relative Mautner property.
By \cite[Corollary 5.2]{BS04}, for every $i \in I$, there exists a generating measure $\mu_i$ on $G_i$ such that $(G_i/P_i,\nu_i)$ is the Furstenberg-Poisson boundary of $(G_i,\mu_i)$, where $\nu_i$ is the (unique) $\mu_i$-stationary measure on $G_i/P_i$ and it is $G_i$-quasi-invariant.
We denote $B_i=G_i/P_i$ and endow it with the quasi-invariant measure $\nu_i$.
We let $B=\prod_I B_i$, endow it with the measure $\nu=\prod_I \nu_i$ and set $\mu=\prod_I \mu_i$.
By \cite[Corollary 3.2]{BS04}, $(B,\nu)$ is the Furstenberg-Poisson boundary of $(G,\mu)$
and by \cite[Theorem 2.7]{BF14} and \cite[Lemma 3.5]{BF18} it is amenable and metrically ergodic as a $G$-space and as a $\Gamma$-space. 

We will prove that $\Gamma$ satisfies condition (a') and (b) from Proposition \ref{top crit 2}. By Margulis normal subgroup theorem \cite[VIII(A), p.\ 259]{Ma91}, the amenable radical of $\Gamma$ is its center, so condition (b) is automatically fulfilled, as was observed in the proof of Corollary \ref{kill center}.
Set $N := L^\infty(B)$, and take a non-invariant $(\Gamma,N)$-von Neumann algebra $(M,E)$. We need to argue that $E$ and $E_g$ are singular, for every $g \in \Gamma \setminus \cZ(\Gamma)$.
For the sake of clarity, we first give the proof in the simply connected setting, and then explain the modifications to make in the general case.

{\it Special case: $\G$ is simply connected. }

In this case the strong approximation theorem (see \cite[Theorem II.6.8]{Ma91}) gives that $G_i = \G_i(k_i)$ and $\Gamma$ is a lattice with dense projections in $G := \prod_{i \in I} G_i$. By Theorem \ref{non-trivial G points}, there exists $i \in I$ such that the $G_i$-algebra $M_i$ in $M$ is non-trivial, $E_{|M_i}$ is not $G_i$-invariant and its image is in $N_i := L^\infty(B_i) \subset N$.

Since the action $\Gamma \actson M$ is ergodic, we note that ${G}_i \actson M_i$ is also ergodic. By Lemma \ref{ergodic extremal}, we find that $E_{|M_i}$ is an extremal $({G}_i,N_i)$-structure map. Proposition \ref{FMM implies singular} then gives that for every $g \in G_i \setminus \cZ(G_i)$, $E_{|M_i}$ and $(E_{|M_i})_g$ are singular.
Observe that the projection map $\Gamma \to G_i$ is injective, 
indeed it coincides with the injection
\[ \Gamma \hookrightarrow \G(K) \to {\bf G}(k_i)={\bf G}_i(k_i). \]
Therefore, the image of $g \in \Gamma \setminus \cZ(\Gamma)$ is in $G_i \setminus \cZ(G_i)$,
thus $E$ and $E_g$ are singular when restricted to $M_i$
and by Remark~\ref{rem:subsing}, it follows that they are singular on $M$. 
This is the desired conclusion.

{\it General case.}

In general unfortunately we don't know that $\Gamma$ is with dense projections, so we may not apply Theorem \ref{non-trivial G points} as such. Nevertheless we show that we can still get the conclusion of this theorem. Once we arrive there, we will just continue the proof as in the simply connected case.

Denote by $(\tM,\tE)$ the induced $(G,\tN)$-structure, as in Example \ref{induced structure}. Since $E$ is not invariant, Lemma \ref{invariant induced} implies that $E_N \circ \tE$ is not $G$-invariant. Since $(B,\nu)$ is the Furstenberg-Poisson boundary of $G$, Example \ref{stationary structure} further implies that the faithful $\mu$-stationary state $\phi = \nu \circ E_N \circ \tE$ is not $G$-invariant on $\tM$. In particular, there exists $i \in I$ such that $\phi$ is not $G_i$-invariant. Gather the factors of $G$ to write it as the product of two groups $G_i \times H_i$. By Lemma \ref{lem:structure-product}, we find that $\phi$ is not $G_i$-invariant on $\tM^{H_i}$. Thanks to the observations in Example \ref{stationary structure}, this amounts to saying that $E_N \circ \tE$ is not $G_i$-invariant on $\tM^{H_i}$. 

At this stage, we don't know a priori that $\tM^{H_i}$ identifies with the $G_i$-algebra in $M$, because we don't know that the projection of $\Gamma$ into $H_i$ is dense. Fortunately, this algebra $\tM^{H_i}$ can be expressed without reference to $H_i$, as the algebra of $\Gamma$-equivariant $L^\infty$-functions $G_i \to M$. We claim that the structure map $E_N \circ \tE$ on this algebra may also be described without appealing to the specific group $H_i$, provided $H_i$ acts metrically ergodically on the Lebesgue space $B_i' := \prod_{j \neq i} B_j$.
In fact, given such an equivariant function $f \in \tM^{H_i}$, the $\Gamma$-invariant function $g \in G_i \mapsto \sigma_g(E(f(g)) \in N$ is essentially constant, by density of the image of $\Gamma$ in $G_i$. We denote by $F(f)$ its essential value.

{\bf Claim.} $F(f) = E_N \circ \tE(f)$, for every $f \in \tM^{H_i}$.

By definition $E_N \circ \tE(f)$ is obtained by viewing the function $f': g \in G \mapsto \sigma_g(E(f(g)) \in N$ as a right $\Gamma$-invariant function and integrating it against the $G$-invariant probability measure on $G/\Gamma$. View $f'$ as an element of $L^\infty(G/\Gamma) \ovt L^\infty(B_i) \ovt L^\infty(B_i')$, which is invariant under the diagonal $H_i$-action (where $H_i$ acts trivially on $B_i$ and metrically ergodically on $B_i'$). Since $\Gamma H_i$ is dense in $G$, $H_i$ acts ergodically on $G/\Gamma$, and hence, by metric ergodicity, we find that $f' \in 1 \ovt L^\infty(B_i) \ovt 1$. So $f'$ is essentially constant; its integral over $G/\Gamma$ is equal to its essential value $y \in L^\infty(B_i)$, i.e. $E_N \circ \tE(f) = y$. Moreover, since $f'$ is essentially constant when viewed as a function over $G$, we find that for almost every $g \in G_i$, $h \in H_i$, $\sigma_{hg}(E(f(g))) = y$. In particular, for almost every $g \in G_i$, $\sigma_g(E(f(g))$ is an $H_i$-invariant element in $N$, equal to $y = E_N \circ \tE(f)$. We thus conclude that $F(f) = E_N \circ \tE(f)$, as claimed.

Thanks to these observations we will replace $H_i$ at our advantage to get the dense projections assumption, and verify that we are still in a situation where $N$ is the Poisson boundary of the (new) ambient group.
Define $H_i' < H_i$ to be the closure of the image of $\Gamma$ in $H_i$ and view $\Gamma$ as a lattice with dense projections inside $G_i \times H_i'$. It is important to observe that $H_i'$ contains the group $\prod_{j \neq i} {\bf G}_j(k_j)^+ $, thanks to the strong approximation theorem (see \cite[Theorem II.6.8]{Ma91}). Thus we may apply \cite[Corollary 5.2]{BS04}, and find a generating measure $\mu_i'$ on $H_i'$ such that the Poisson boundary of $(H_i',\mu_i')$ can be identified with $B_i'$, as a Lebesgue $H_i'$-space.
By \cite[Corollary 3.2]{BS04}, the Lebesgue space $B = B_i \times B_i',$ is the Furstenberg-Poisson boundary of $G_i \times H_i'$, for the measure $\mu_i \otimes \mu_i'$. We can now apply Theorem \ref{continuous elts} to $\Gamma < G_i \times H_i'$ with the $(\Gamma,N)$-structure $(M,E)$. We obtain an identification between the $G_i$-algebra $M_i$ and the algebra of $H_i'$-invariant elements $L^\infty(G_i \times H_i',M)^{H_i' \times \Gamma}$ which intertwines the natural $(G_i,N_i)$-structure maps. By the previous paragraph, we know that the later algebra $L^\infty(G_i \times H_i',M)^{H_i' \times \Gamma}$ together with its $(G_i,N_i)$-structure map identifies with $(\tM^{H_i},E_N \circ \tE)$. Since $E_N \circ \tE$ is not $G_i$-invariant on $\tM^{H_i}$, we conclude that $E$ is not $G_i$-invariant on $M_i$.

As announced we thus conclude that there is an index $i$ and a $\Gamma$-invariant von Neumann subalgebra\footnote{which really is the algebra of all $G_i$-continuous elements in $M$} $M_i \subset M$ on which the $\Gamma$-action extends to a continuous action $G \curvearrowright M_i$ that factors through the projection map $G \to G_i$, and on which $E$ is not $\Gamma$-invariant. We can now finish the proof as in the simply connected case.
\end{proof}

We end this subsection by proving Theorem~\ref{thm:SD}.

\begin{proof}[Proof of Theorem~\ref{thm:SD}]
We denote $\Lambda=\SL_n(\Z) \ltimes \Z^n$ and $\Gamma=\SL_n(\Z)$.
We let $B$ the flag manifold associated with $G=\SL_n(\R)$ and check that (a) and (b) of Proposition~\ref{prop:chmcrit3}
are satisfied.
Using Fourier transform we identify $\Char(\Rad(\Lambda))$ with $\Prob(T^n)$.
Then (b) follows from the main result of \cite{Fu98}. 
The proof of property (a) follows from \cite[Theorem B]{BH19}
by combining Proposition \ref{FMM implies singular} with Lemma \ref{freeness1} as above.
\end{proof}

\subsection{Lattices in product of trees}
\label{ss:lpt}

This subsection is devoted to the proof of Theorem~\ref{thm:Tree}.

\begin{prop}\label{freeness2}
Fix $n\geq 2$ and natural numbers $p_1,\ldots,p_n,q_1,\dots,q_n>1$.
For each $1\leq i \leq n$, let $T_i$ be a $(p_i+1,q_i+1)$-biregular simplicial tree and let $\Gamma < \Aut^+(T_1) \times \cdots \times \Aut^+(T_n)$ be a cocompact lattice.
We look at a projection onto a factor, say the first factor. Endow $\partial T_1$ with its unique $\Aut^+(T_1)$-invariant measure class, and look at the action of $\Gamma$ on $\partial T_1$. If the first projection is injective on $\Gamma$, then for every $g \in \Gamma\setminus \{e\}$, the fixed point set in $\partial T_1$ has measure $0$.
\end{prop}

To prove the proposition we need the following lemma.

\begin{lem} \label{lem:nullpartial}
Fix integers $p,q>1$, a $(p+1,q+1)$-biregular simplicial tree $T$ and a proper subtree $T'\subsetneq T$.
Assume that the subgroup $H \leq \Aut(T)$ which stabilizes $T'$ acts on it cocompactly.
Then $\partial T'$ is a null set of  $\partial T$, where $\partial T$ is endowed with the unique $\Aut(T)$-invariant measure class.
In particular, this measure class on $\partial T$ is non-atomic.
\end{lem}
\begin{proof}
We fix a vertex $o\in T'$ and consider the space $R$ consisting of rays in $T$ emanating at $o$
and the natural surjection $\pi:R \to \partial T$.
Endowing $R$ with the unique probability measure $\mu$ which is invariant under $\Stab(o)<\Aut(T)$,
this map is measure class preserving.
We thus need to show that the subset $R'=\pi^{-1}(\partial T')$ is a null set in $R$.

By the assumption that $T'\neq T$ there exist adjacent vertices $u,v\in T$ such that $u\in T'$, $v\notin T'$.
Without loss of the generality we assume that the degree of $u$ is $p+1$.
Setting 
\[ A=\{x\in R \mid x(k)\in H u\mbox{ for infinitely many values of $k$}\} \]
we easily see that $A\subset R'$.
By the fact that $H$ acts cocompactly on $T'$ and the law of large numbers we also have that $A$ is conull in $R'$,
thus $\mu(A)= \mu(R')$.
Writing $A$ as the descending intersection
$A=\bigcap_{n\in \mathbb{N}} A_n$ of
\[ A_n=\{x\in R\mid x(k)\in H u\mbox{ for at least $n$ values of $k$}\}, \]
we have that $\mu(A)= \lim_n \mu(A_n)$.
Since for all $n \geq 1$, $\mu(A_{n+1})\leq \left(\frac{p-1}{p}\right)\mu(A_n)$
we get that $\mu(A_n)\leq \left(\frac{p-1}{p}\right)^{n-1}$,
thus indeed $\partial T'$ is a null set in $\partial T$.

The last sentence of the proposition follows by considering the special case where $T'$ is a geodesic in $T$.
\end{proof}

\begin{proof}[Proof of Proposition \ref{freeness2}]
We fix a non-trivial element $g\in \Gamma$ and set $F=\Fix(g)$.
We assume as we may that $F$ has at least three points.
Let $T'_{1}$ be the convex hull in $T_{1}$ of $F$.
Then $T'_{1}$ is non-empty, it coincides with the set of fixed points of $g$ in $T_{1}$
and $F=\partial T'_{1}$.
Let $Z<G$ be the centralizer of $g$
and note that $T'_{1}$ is $Z$ invariant.
We claim that the $Z$-action on $T'_{1}$ is cocompact.
From this claim we will get by Lemma~\ref{lem:nullpartial} that $F=\partial T'_{1}$ is a null
set in $\partial T_{1}$,
thus proving the proposition.

We endow $X = T_1 \times \cdots \times T_n$ with the $L^2$-product metric and note that this is a CAT(0) space.
We consider the displacement function $D:X\to [0,\infty)$, $D(x)=d(gx,x)$ and let $Y\subset X$ be its minset,
that is setting $d_0=\min_{x\in X} d(gx,x)$, $Y=D^{-1}(d_0)$.
Note that the image of $D$ is discrete in $[0,\infty)$, as the $\Gamma$ action on $X$ is simplicial, thus $D$ attends its minimum $d_0$ and $Y$ is a 
closed convex subset of $X$.
By a result of Kim Ruane, \cite[Theorem 3.2 and Remark 1]{Ru99},
the action of $Z$ on $Y$ is cocompact.
In particular, the $Z$-action on the image of $Y$ under the projection $X \to T_{1}$ is cocompact.
We are done by observing that this image is exactly the minset of $g$ in $T_{1}$, that is the tree of $g$-fixed points $T'_{1}$. 
\end{proof}

\begin{proof}[Proof of Theorem~\ref{thm:Tree}]
By \cite[Lemma 3.1.1, Proposition 3.1.2]{BM00}, the $2$-transitivity assumption implies that for every $i$, every closed normal subgroup of $G_i$ is co-compact. This $2$-transitivity also implies that $G_i$ is non-amenable, and hence it has no non-trivial amenable normal closed subgroup. So $\Gamma$ has trivial amenable radical. 

Let us argue that each projection map $\Gamma \to G_i$ is injective. Indeed, the kernel of such a projection map is equal to $\Gamma \cap H_i$, where $H_i = \prod_{j \neq i} G_j$. It is a closed subgroup of $H_i$, which is normalized by the projection of $\Gamma$ on $H_i$. So by the dense projection assumption, $\Gamma \cap H_i$ is a normal closed subgroup of $H_i$. Since every non-trivial normal subgroup of each factor $G_j$, $j \neq i$, is co-compact in $G_j$, this normal subgroup is either trivial or it contains a co-compact normal closed subgroup $G_j'$ of some $G_j$, $j \neq i$. In this case, $\Gamma$ contains $G_j'$. Further, $\Gamma/G_j'$ is a lattice with dense projections inside $(G_j/G_j') \times \prod_{i \neq j} G_i$. The only way this can happen is if the compact factor $G_j/G_j'$ is trivial. In this case, $G_j = G_j'$ is discrete, which contradicts the $2$-transitivity assumption, and the fact that $T_i$ is thick.

We set for every $i \in I$, $B_i=\partial T_i$ endowed with the unique $G_i$-invariant measure class and let $B=\prod B_i$.
For every $i \in I$, we fix a point in $\partial T_i$ and let $P_i$ be its stabilizer.
By this we identify $B_i=G_i/P_i$.
By Example~\ref{Mautner trees}, $P_i$ is stably self-normalizing and it has the relative Mautner property in $G_i$.
By \cite[Theorem 5.1]{BS04} we have that $(B_i, \nu_i)$ is the Furstenberg-Poisson boundary of $G_i$ for some generating measure $\mu_i$ on $G_i$
and by \cite[Corollary 3.2]{BS04}, $(B,\nu)$ is the Furstenberg-Poisson boundary of $G$ for the measure $\mu=\prod \mu_i$.
By \cite[Theorem 2.7]{BF14}, $B$ is amenable and metrically ergodic $G$-space
and by \cite[Lemma 3.5]{BF18} it is amenable and metrically ergodic $\Gamma$-space.
Therefore, by Proposition~\ref{top crit}, it is enough to verify condition (a) of Proposition~\ref{prop:chmcrit3}.

We now fix a separable, ergodic, faithful $(\Gamma,L^\infty(B))$-von Neumann algebra $(M,E)$ which is not $\Gamma$-invariant
and argue to show that it is $\Gamma$-singular.
By Theorem \ref{non-trivial G points}, we find an index $i \in I$ such that the $G_i$-algebra $M_i$ in $M$ is non-trivial, and such that $E_{|M_i}$ is not $G_i$-invariant. Since the action $\Gamma \actson M$ is ergodic, we note that $G_i \actson M_i$ is also ergodic. By Lemma \ref{ergodic extremal}, we find that $E_{|M_i}$ is an extremal $(G_i,L^\infty(B_i))$-structure map. 
We combine Proposition \ref{FMM implies singular} with our freeness result Proposition \ref{freeness2} and find that $E_{|M_i}$ is $\Gamma_i$-singular, where $\Gamma_i$ is the projection of $\Gamma$ into ${G}_i$. 
As the projection map $\Gamma \to G_i$ is injective, $E_{|M_i}$ is $\Gamma$-singular. From the characterizations of singular ucp maps we gave, this implies that $E$ is $\Gamma$-singular, as desired.
\end{proof}


\section{Proofs of charfiniteness}\label{sec:proofs}


In this section, we prove the second half of Theorem~\ref{thm:AG} and Theorem~\ref{thm:Zp}.

\subsection{Finite dimensional unitary representations}
\label{ss:fdur}

In this subsection, we prove the following proposition, which is well known to experts.

\begin{prop} \label{prop:ffd}
Let $K$ be a global field and ${\bf G}$ a connected non-commutative $K$-almost simple $K$-algebraic group.
Let $\Gamma\leq {\bf G}(K)$ be an $S$-arithmetic subgroup of higher rank.
If either $S$ is finite or ${\bf G}$ is simply connected then $\Gamma$ has a finite number of isomorphism types of
unitary representation at each finite dimension.
\end{prop}

We will use heavily the results of \cite[Chapter VIII]{Ma91}
and also rely on \cite[Section 5]{Sh99}.
For the terminology regarding arithmetic groups used in the proof, see Definition~\ref{def:arith}.

\begin{proof}
We first note that if $\Gamma$ is of a simple type and of higher rank then it has property (T),
which clearly implies the result.
Thus we assume as we may that $\Gamma$ is of a product type.
Next, we observe that if
$\Lambda$ has the property of having a finite number of isomorphism types of
unitary representation at each finite dimension and $\Lambda\to \Gamma$ is a homomorphism with a finite kernel and 
finite index image then also $\Gamma$ has this property.
Therefore we assume as we may that ${\bf G}$ is simply connected even in case $S$ is finite.
Indeed, in this case letting $\tilde{\bf G}$ be the simply connected cover of ${\bf G}$ and letting $\Lambda$ be the preimage of 
$\Gamma$ under the covering map $\tilde{\bf G}(K)\to {\bf G}(K)$, we have that $\Lambda\leq \tilde{\bf G}(K)$ is an $S$-arithmetic subgroup of higher rank and $\Lambda\to \Gamma$ is a homomorphism with a finite kernel and 
finite index image.
We fix $n$ and argue to show that $\Gamma$ has a finite number of $\U(n)$-conjugacy classes of homomorphisms into $\U(n)$. This fact would easily follow from \cite[Theorem 5.7]{Sh99} in the case where $S$ is finite. However, such a  statement is badly behaved under inductive limits. For this reason we need to be more accurate and invoke superrigidity techniques of Margulis.

Let us say that such a homomorphism $\Gamma \to \U(n)$ is finite if it has a finite image.

\begin{claim}
$\Gamma$ has a finite number of $\U(n)$-conjugacy classes of finite homomorphisms into $\U(n)$.
\end{claim}

We will prove the claim later, first finishing the proof assuming the claim.
We first note that we may assume that $K$ is of characteristic zero,
as if $K$ is of positive characteristic then every homomorphism $\rho:\Gamma\to \U(n)$ is finite
by \cite[VIII(C), p.\ 259]{Ma91}.
Indeed, if $\rho$ had an infinite image, upon setting $\ell=\R$ and denoting by ${\bf H}$ the identity component of the 
Zariski closure of $\rho(\Gamma)$, we would get a field extension $K\to \ell$, thus a contradiction.
We thus may apply \cite[VIII(B)(iii), p.\ 258]{Ma91} for $\ell=\R$
and letting ${\bf H}$ be the $n$-dimensional real unitary group, thus ${\bf H}(\ell)=\U(n)$.
It follows that any $\rho:\Gamma\to \U(n)$ is
of the form $\rho=\phi\cdot \nu$ where
$\phi:\Gamma\to \U(n)$ is obtained by a composition
\[ \Gamma \to {\bf G}(K)\simeq R_{K/\Q}{\bf G}(\Q) \to R_{K/\Q}{\bf G}(\ell) \to {\bf H}(\ell)=\U(n), \]
where $R_{K/\Q}{\bf G}(\ell) \to {\bf H}(\ell)$ is the $\ell$-points evaluation of an $\ell$-algebraic morphism
$R_{K/\Q}{\bf G}\to {\bf H}$
and $\nu:\Gamma\to \U(n)$ is a finite homomorphism whose image commutes with $\phi(\Gamma)$.
Since the number of $\ell$-algebraic morphisms
$R_{K/\Q}{\bf G}\to {\bf H}$ is finite, we are done by the claim.

Next we prove the claim.
By \cite[VIII(A), p.\ 258]{Ma91}\footnote{We note that the assumption that ${\bf G}$ is simply connected is missing in this reference.
This is certainly a typo, as this assumption is used in its proof. Of course when $S$ is finite this doesn't matter, but 
when $S$ is infinite this assumption is necessary, as can be seen for example by the natural morphism $\PGL_2(\Q)\to\Q^*/(\Q^*)^2$
determined by the determinant morphism.} 
it is enough to show that there exists a non-central element $g\in \Gamma$ which is in the kernel of
all finite morphisms $\Gamma\to \U(n)$.
Indeed, letting $N\le \Gamma$ be the normal subgroup generated by $g$, we have that $\Gamma/N$ is finite and since every 
finite morphism $\Gamma\to \U(n)$ factors via $\Gamma/N$ there is only a finite number of isomorphism types of those.
We can find a finite subset $S_0\subset S$ such that the $S_0$-arithmetic subgroup $\Gamma_0=\Gamma\cap \Lambda_{S_0}$
(see Definition~\ref{def:arith}) is still of a product type. 
We thus assume as we may that $S$ is finite.
By the fact that ${\bf G}$ is simply connected, the strong approximation theorem (see \cite[Theorem II.6.8]{Ma91}) implies that $\Gamma$ is a lattice with dense projections in 
$G=\prod_{i\in I} G_i$ where for each $i\in I$, $G_i$ is the $k_i$-points of the $K$-algebraic group ${\bf G}$ for some local field extension $k_i$ of $K$.
By Proposition~\ref{prop:trivcharlie}, the groups $G_i$ have no non-trivial characters and it follows that they have no non-trivial finite dimensional unitary representations.
It then follows by \cite[Theorem 5.7]{Sh99} that $\Gamma$ has a finite number of isomorphism types of
homomorphism into $\U(n)$, and in particular a finite number of finite ones.
It follows that all finite homomorphisms into $\U(n)$ factor via one finite quotient, thus we can pick a non-central element $g\in \Gamma$
which is in its kernel.
This finishes the proof.
\end{proof}

A similar statement also holds for lattices in product of groups of automorphisms of trees.

\begin{prop} \label{prop:ffdlt}
Let $n\geq 2$. For every $i=1,\ldots,n$, let $T_i$ be a bi-regular tree and denote by $\Aut^+(T_i)$ the group of bi-coloring preserving automorphisms of $T_i$. Let $\Gamma < \Aut^+(T_1) \times \cdots \times \Aut^+(T_n)$ be a cocompact lattice. Denote by $G_i$ the closure of the image of $\Gamma$ in $\Aut^+(T_i)$ and assume that $G_i$ is $2$-transitive on the boundary $\partial T_i$. Further assume that $\Gamma < G_1 \times \cdots \times G_n$ is with dense projections. Then $\Gamma$ has a finite number of isomorphism types of
unitary representation at each finite dimension.
\end{prop}

\begin{proof}
We note that $G=G_1\times\cdots \times G_n$ is totally disconnected, hence it has an open kernel for every continuous homomorphism
into a Lie group.
It is a standard fact that the groups $G_i$ are just non-compact, thus we get that every open normal subgroup of $G$ 
is of finite index.
It follows that every continuous homomorphism of $G$ into the Lie group $\U(d)\ltimes \mathbb{C}^d$ has a finite image,
thus the corresponding cohomology group $H^1(G,\mathbb{C}^d)$ vanishes.
Since for a finite dimensional unitary representation $\Gamma\to \U(d)$ we have
$H^1(\Gamma,\mathbb{C}^d)=\bar{H}^1(\Gamma,\mathbb{C}^d)$ (for the definition of the reduced cohomology $\bar{H}^1$,
see \cite[Definition~1.7]{Sh99}),
we conclude by \cite[Theorem 4.1]{Sh99} that $H^1(\Gamma,\mathbb{C}^d)=0$.

By the fact that $\Gamma$ acts properly and cocompactly on a locally finite simplicial complex (a product of trees), it is finitely generated.
We may thus consider the compact representation variety $\Hom(\Gamma,\U(d))$.
The previous paragraph allow us to conclude that the $\U(d)$ action on it by postcomposition with conjugation has open orbits, see \cite{We63}.
It follows by compactness that there are only finitely many such orbits,
thus indeed, $\Gamma$ has a finite number of isomorphism types of
unitary representation at dimension $d$.
\end{proof}

We are now ready to prove Theorem~\ref{thm:Zp}.

\begin{proof}[Proof of Theorem~\ref{thm:Zp}]
Fix a non-empty set of primes $S$ and set $\Gamma=\SL_2(\mathbb{Z}_S)$.
We have seen in Proposition~\ref{SAF}
that $\Gamma$ is charmenable so we are left to verify properties (3)-(5) in Definition~\ref{def:charmenable}.
Property (3) follows from \cite[VIII(A), p.\ 258]{Ma91},
property (4) was verified in Proposition~\ref{prop:ffd}
and property (5) follows from \cite[Theorem 2.6]{PT13}.
So indeed, $\Gamma$ is charfinite.
\end{proof}

\subsection{Char-(T) and the proof of Theorem~\ref{thm:AG}}
\label{sec:upgrade}

Recall that an lcsc group $G$ has property (T) if and only if every amenable representation of $G$ contains a finite dimensional sub-representation, see \cite[Theorem~1.1]{BV91}. 

\begin{defn}
An lcsc group $G$ is said to have 
property {\em char-}(T) if for every amenable character the associated GNS representation contains a finite dimensional sub-representation.
\end{defn}

\begin{prop} \label{prop:charTcrit}
Let $\Gamma$ be a lattice with dense projections in $G = G_1\times G_2$.
Assume that $G_1$ has property {\em (T)} and $\Char(G_2) = \{1\}$.
Then $\Gamma$ has char-{\em (T)}.
\end{prop}

\begin{proof}
Let $\phi$ be an amenable character of $\Gamma$, and denote by $(\pi,H,\xi)$ the corresponding GNS triple. 
We need to prove that $H$ contains a non-zero finite dimensional invariant subspace. We will argue that $(H \ot \oH,\pi \ot \overline\pi)$ has non-zero invariant vectors. Note that $\pi \ot \overline\pi$ is a tracial representation of $\Gamma$, in the sense that $(\pi \ot \overline\pi)(\Gamma)\dpr$ has a normal faithful trace, implemented by the vector $\xi \ot \overline\xi$. Moreover $\pi \ot \overline\pi$ has almost invariant vectors.

Denote by $(\tH,\tpi)$ the induced unitary representation of $(H \ot \oH,\pi \ot \overline\pi)$ to $G$. Then this representation has almost invariant vectors, and since $G_1$ has property (T), $\tH$ must contain non-trivial $G_1$-invariant vectors. 
By assumption, $\iota(\Gamma)$ is dense in $G_2$,  where $\iota: \Gamma \to G_2$ is the restriction of the projection map. 
We observed in Proposition \ref{continuity rep} that $\tH^{G_1}$ is naturally identified with the $\iota$-continuity space of $H \ot \oH$.
In other words, we have found a non-trivial subspace $K \subset H \ot \oH$ on which the representation $\pi \ot \overline\pi$ extends to a continuous representation of $G_2$. Since the representation of $\Gamma$ on $H \ot \oH$ is tracial, so is the restricted representation on $K$. Thus the continuous extension $\rho: G_2 \to \cU(K)$ is a tracial representation. 
But $\Char(G_2)$ is trivial so every continuous group homomorphism from $G_2$ into the unitary group of a tracial von Neumann algebra is trivial. Therefore $\rho$ is trivial, which implies that $H \ot \oH$ contains invariant vectors, as desired.
\end{proof}

\begin{prop} \label{prop:charTAG}
Let $K$ be a global field and ${\bf G}$ a connected non-commutative $K$-almost simple $K$-algebraic group.
Let $\Gamma\leq {\bf G}(K)$ be an $S$-arithmetic subgroup of a product type.
Assume further that there exists an absolute value $v$ such that ${\bf G}(K_v)$ has property {\em (T)}.
If either $S$ is finite or ${\bf G}$ is simply connected then 
$\Gamma$ has char-{\em (T)}.
\end{prop}

\begin{proof}
We begin as in the proof of Proposition~\ref{prop:ffd}.
We first note that if $\Gamma$ is of a simple type and of higher rank then it has property (T),
which clearly implies the result.
Thus we assume as we may that $\Gamma$ is of a product type.
Next, observe that if
$\Lambda$ has char-(T) and $\Lambda\to \Gamma$ is a homomorphism with a finite kernel and 
finite index image then also $\Gamma$ has char-(T).
Therefore we assume as we may that ${\bf G}$ is simply connected even in case $S$ is finite.
Indeed, in this case letting $\tilde{\bf G}$ be the simply connected cover of ${\bf G}$ and letting $\Lambda$ be the preimage of 
$\Gamma$ under the covering map $\tilde{\bf G}(K)\to {\bf G}(K)$, we have that $\Lambda\leq \tilde{\bf G}(K)$ is an $S$-arithmetic subgroup of a higher rank and $\Lambda\to \Gamma$ is a homomorphism with a finite kernel and 
finite index image.
As usual, view $\Gamma$ as a lattice in the corresponding restricted product of all almost simple factors over all local completions of $K$ in which 
the image of $\Gamma$ is unbounded.
We set $G_1= {\bf G}(K_v)$ and denote the restricted product of all other factors by $G_2$.
By the strong approximation theorem (see \cite[Theorem II.6.8]{Ma91}), $\Gamma$ is a lattice with dense projections in $G = G_1\times G_2$.
By Proposition~\ref{prop:trivcharlie}, the group $G_2$ has no non-trivial characters and by assumption $G_1$ has (T).
It follows by Proposition~\ref{prop:charTcrit} that $\Gamma$ has char-(T).
\end{proof}

The proof of Theorem~\ref{thm:AG} now follows similarly to the proof Theorem~\ref{thm:Zp}.

\begin{proof}[Proof of Theorem~\ref{thm:AG}]
We have seen in Proposition~\ref{SAF}
that $\Gamma$ is charmenable so we are left to verify properties (3)-(5) in Definition~\ref{def:charmenable}.
Property (3) follows from \cite[VIII(A), p.\ 258]{Ma91},
property (4) was verified in Proposition~\ref{prop:ffd}
and property (5) follows from Proposition~\ref{prop:charTAG}.
So indeed, $\Gamma$ is charfinite.
\end{proof}

Combining \cite[Theorem B]{BH19} with Corollary \ref{kill center} and using property (T), we also get the following result. 

\begin{cor}\label{cor:BH19}
Let $G$ be a simple Lie group of higher rank with finite center and let $\Gamma$ be a lattice in $G$.
Then $\Gamma$ charfinite.
\end{cor}


\bibliographystyle{plain}

\end{document}